\documentclass[oneside,english]{siamltex}
\usepackage[T1]{fontenc}
\usepackage[latin9]{inputenc}
\usepackage{geometry}
\usepackage{xcolor}
\usepackage{pdfcolmk}
\usepackage{babel}
\usepackage{mathrsfs}
\usepackage{amstext}
\usepackage{amssymb}
\usepackage{graphicx}
\usepackage[numbers]{natbib}
\PassOptionsToPackage{normalem}{ulem}
\usepackage{ulem}
\usepackage[unicode=true,
 bookmarks=true,bookmarksnumbered=false,bookmarksopen=false,
 breaklinks=false,pdfborder={0 0 1},backref=false,colorlinks=false]
 {hyperref}

\makeatletter

\providecolor{lyxadded}{rgb}{0,0,1}
\providecolor{lyxdeleted}{rgb}{1,0,0}

\newcommand\eqref[1]{(\ref{#1})}

\makeatletter
\let\eqref\relax
\usepackage{amsmath}
\usepackage{mathtools}
\newcommand{\sign}{\text{ sgn}}
\newtheorem{cond}{Condition}

\newtheorem{rem}{Remark}
\newtheorem{exmp}{Example}

\makeatother

\begin{document}

\title{Stability of steady states for a class of flocculation equations
with growth and removal}

\author{Inom Mirzaev\thanks{Department of Applied Mathematics, University of Colorado, Boulder, CO 80304-0526, USA.}
\and David M. Bortz$^{*}$\thanks{Corresponding author ({\tt dmbortz@colorado.edu})}}
\maketitle
\begin{abstract}
Flocculation is the process whereby particles (i.e., \emph{flocs})
in suspension reversibly combine and separate. The process is widespread
in soft matter and aerosol physics as well as environmental science
and engineering. We consider a general size-structured flocculation
model, which describes the evolution of flocs in an aqueous environment.
Our work provides a unified treatment for many size-structured models
in the environmental, industrial, medical, and marine engineering
literature. In particular, our model accounts for basic biological
phenomena in a population of microorganisms including growth, death,
sedimentation, predation, renewal, fragmentation and aggregation.
Our central goal in this paper is to rigorously investigate the long-term
behavior of this generalized flocculation model. Using results from
fixed point theory we derive conditions for the existence of continuous,
non-trivial stationary solutions. We further apply the principle of
linearized stability and semigroup compactness arguments to provide
sufficient conditions for local exponential stability of stationary
solutions as well as sufficient conditions for instability.

The end results of this analytical development are relatively simple
inequality-criteria which thus allows for the rapid evaluation of
the existence and stability of a non-trivial stationary solution.
To our knowledge, this work is the first to derive precise existence
and stability criteria for such a generalized model. Lastly, we also
provide an illustrating application of this criteria to several flocculation
models.\end{abstract}



\section{Introduction}

\emph{Flocculation} is the process whereby particles (i.e., flocs)
in suspension reversibly combine and separate. The process is widespread
in soft matter and aerosol physics as well as environmental science
and engineering. A popular model for flocculation is a 1D nonlinear
partial integro-differential equation which describes the time-evolution
of the particle size number density. In the engineering literature,
this equation can be derived using a so-called population balance
equation (PBE) framework and we direct the interested reader to the
book by Ramkrishna \citep{Ramkrishna2000} for more on this framework.

Previous analytical work on these models focused on classes of flocculation
equations that did not allow for the \emph{vital dynamics} (i.e.,
birth and death) of individual particles. These phenomena are obviously
critical features in the modeling of microbial flocculation. Accordingly,
in this work we consider a general size-structured flocculation model
which accounts for growth, aggregation, fragmentation, surface erosion
and sedimentation. The variable $p(t,\,x)$ denotes the number density
of flocs of size $x$ at time $t$, and for a given interval $U\subseteq(x_{0},\,x_{1})$,
the function $\chi_{U}$ represents the characteristic function of
the interval $U$. A floc is assumed to have a minimum and maximum
size $x_{0}$ and $x_{1}$ such that $0\le x_{0}<x_{1}\le\infty$.The
equations for the microbial flocculation model can be written as 
\begin{eqnarray}
\partial_{t}p & = & \mathcal{F}[p]\label{eq: agg and growth model}
\end{eqnarray}
where
\[
\mathcal{F}[p]:=\mathcal{G}[p]+\mathcal{A}[p]+\mathcal{B}[p],
\]
$\mathcal{G}$ denotes growth
\begin{equation}
\mathcal{G}[p]:=-\partial_{x}(gp)-\mu(x)p(t,\,x)\,,\label{eq:Growth}
\end{equation}
$\mathcal{A}$ denotes aggregation
\begin{alignat}{1}
\mathcal{A}[p] & :=\frac{1}{2}\chi_{[2x_{0},\,x_{1})}(x)\int_{x_{0}}^{x-x_{0}}k_{a}(x-y,\,y)p(t,\,x-y)p(t,\,y)\,dy\nonumber \\
 & \quad-\chi_{[x_{0},\,x_{1}-x_{0}]}p(t,\,x)\int_{x_{0}}^{x_{1}-x}k_{a}(x,\,y)p(t,\,y)\,dy\label{eq:Aggregation}
\end{alignat}
and $\mathcal{B}$ denotes breakage 
\begin{equation}
\mathcal{B}[p]:=\chi_{(x_{0},\,x_{1}-x_{0}]}(x)\int_{x}^{x_{1}}\Gamma(x;\,y)k_{f}(y)p(t,\,y)\,dy-\frac{1}{2}\chi_{[2x_{0},\,x_{1})}k_{f}(x)b(t,\,x)\,.\label{eq:Breakage}
\end{equation}
The boundary condition is traditionally defined at the smallest size
$x_{0}$ and the intial condition is defined at $t=0$

\[
g(x_{0})p(t,\,x_{0})=\int_{x_{0}}^{x_{1}}q(x)p(t,\,x)dx,\quad p(0,\,x)=p_{0}(x)\in L^{1}(I)\,.
\]

Note that \emph{in vivo}, there are no flocs of size zero and the
flocs cannot grow indefinitely, so the only biologically realistic
case is $0<x_{0}<x_{1}<\infty$. However, when $x_{0}>0$, the characteristic
functions appearing in equations (\ref{eq:Aggregation}) and (\ref{eq:Breakage})
make our theoretical development rather cumbersome. Hence, for the
sake of convenience, in this paper we consider the case $x_{0}=0<x_{1}<\infty$,
and postpone the analysis of the case with $0<x_{0}<x_{1}=\infty$
for our future papers. Therefore, we will denote the closed interval
by $I:=[0,\,x_{1}]$ and will make extensive use of this interval
in our development. We carry out the analysis of this work (unless
otherwise specified) on the space of absolutely integrable functions
on $I$, denoted by $L^{1}(I)$. We also note that well-posedness
of the flocculation model on this space has been established by Banasiak
and Lamb \citep{Lamb2009}.

\subsection{Background and model terms\label{sub:Model-equations}}

In this section, we will provide a brief background and overiew of
the individual terms in the general flocculation equation above.

To begin, the Sinko-Streifer \citep{Sinko1967} terms in (\ref{eq:Growth})
correspond to the growth and removal of flocs, respectively. The function
$g(x)$ represents the average growth rate of the flocs of size $x$
due to mitosis, and the coefficient $\mu(x)$ represents a size-dependent
removal rate due to gravitational sedimentation and cell death. Specifically,
when an individual cell in the floc of size $x$ divides into daughter
cells, the new cells can remain with the floc, contributing in a increase
in its total size. Conversely, a daughter cell can also leave the
floc to form a new single-cell floc. This second case is modeled by
McKendrick-von Foerster type renewal boundary conditions,
\begin{equation}
g(x_{0})p(t,\,x_{0})=\int_{x_{0}}^{x_{1}}q(x)p(t,\,x)dx\,,\label{eq:boundary condition}
\end{equation}
where the renewal rate $q(x)$ represents the number of new cells
that leave a floc of size $x$ and enter the single cell population.
We note that this boundary condition could also be used to model the
surface erosion of flocs, where single cells are eroded off the floc
and enter single cell population. The well-posedness and stability
of equilibrium solutions of the Sinko-Streifer equations has been
extensively studied by many researchers using a wide variety of mathematical
conditions \citep{Gurtin1972,Gurtin1979,BanksKappel1989,Equilibria1983,Pruss1983,Diekmann1984c}.
For numerical simulation of the model, a convergent numerical scheme
has been proposed in \citep{BanksKappel1989}, and inverse problems
for estimation of the parameters of the model have been discussed
in \citep{Banks1987,BanksKunisch1989,Fitzpatrick1993}. 

The aggregation of flocs into larger ones is modeled in (\ref{eq:Aggregation}),
by the Smoluchowski coagulation equation. The function $k_{a}(x,\,y)$
is the aggregation kernel, which describes the rate with which the
flocs of size $x$ and $y$ agglomerate to form a floc of size $x+y$.
This equation has been widely used, e.g., to model the formation of
clouds and smog in meteorology \citep{Pruppacher1980}, the kinetics
of polymerization in biochemistry \citep{Ziff1980}, the clustering
of planets, stars and galaxies in astrophysics \citep{Makino1998},
and even schooling of fish in marine sciences \citep{Niwa1998}. The
equation has also been the focus of considerable mathematical analysis.
For the aggregation kernels satisfying the inequality $k_{a}(x,\,y)\le1+x+y$,
existence of mass conserving global in time solutions were proven
\citep{Dubovski1996,Fournier2005,Menon2005} (for some suitable initial
data). conversely, for aggregation kernels satisfying $(xy)^{\gamma/2}\le k_{a}(x,\,y)$
with $1<\gamma\le2$, it has been shown that the total mass of the
system blows up in a finite time (referred as a \emph{gelation time})
\citep{Escobedo2002a}. For a review of further mathematical results,
we refer readers to review articles by Aldous \citep{Aldous1999},
Menon and Pego \citep{Menon2006}, and Wattis \citep{Wattis2006a}
and the book by Dubovskii \citep{Dubovskii1994}. Lastly, although
the Smoluchowski equation has received substantial theoretical work,
the derivation of analytical solutions for many realistic aggregation
kernels has proven elusive. Towards this end, many discretization
schemes for numerical simulations of the Smoluchowski equations have
been proposed, and we refer interested readers to the review by Bortz
\citep[\S 6]{Bortz2014}.

The breakage of flocs due to fragmentation is modeled by the terms
in (\ref{eq:Breakage}), where the fragmentation kernel $k_{f}(x)$
calculates the rate with which a floc of size $x$ fragments. The
integrable function $\Gamma(x;y)$ represents the post-fragmentation
probability density of daughter flocs for the fragmentation of the
parent flocs of size $y$. The post-fragmentation probability density
function $\Gamma$ is one of the least well-understood terms in the
flocculation model. Many different forms are used in the literature,
among which normal and log-normal densities are the most common \citep{HanEtal2003AICHEJ,Spicer1996}.
Recent modeling and computational work suggests that normal and log-normal
forms for $\Gamma$ are not correct and that a form closer to an $\arcsin(x;y)$
density would be more accurate \citep{BortzByrne2011,ByrneEtal2011pre}.
However, in this work we do not restrict ourselves to any particular
form of $\Gamma$, and instead simply assume that the function $\Gamma$
satisfies the mass conservation requirement. In other words, all the
fractions of daughter flocs formed upon the fragmentation of a parent
floc sum to unity,
\[
\int_{x_{0}}^{y}\Gamma(x;\,y)\,dx=1\text{ for all }y\in(x_{0},\,x_{1}].
\]

\subsection{Overview of model assumptions}

The flocculation model, presented in (\ref{eq: agg and growth model}),
is a generalization of many mathematical models appearing in the size-structured
population modeling literature and has broad applications in environmental,
industrial, medical, and marine sciences. For example, when the fragmentation
kernel is omitted, $k_{f}\equiv0$, the flocculation model reduces
to algal aggregation model used to describe evolution of phytoplankton
community \citep{AcklehFitzpatrick1997}. When the removal and renewal
rates are set to zero, the flocculation model simplifies to a model
used to describe the proliferation of \emph{Klebsiella pneumoniae
}in a bloodstream \citep{Bortz2008}. Furthermore, the flocculation
model, with only growth and fragmentation terms, was used to investigate
the elongation of prion polymers in infected cells \citep{Gabriel2011,JAUFFRET2010}. 

The flocculation model in this form (\ref{eq: agg and growth model})
was first considered by Banasiak and Lamb in \citep{Lamb2009}, where
they employed the flocculation model to describe the dynamical behavior
of phytoplankton cells. The authors showed that under some conditions
the flocculation model is well-posed, i.e., there exist a unique,
global in time, positive solution for every absolutely integrable
initial distribution. For the case $x_{1}=\infty$, Banasiak \citep{Africa2011}
establishes that for certain range of parameters, the solutions of
the flocculation model blow up in finite time. Nevertheless, to the
best of our knowledge, for the case $x_{1}<\infty$ the long-term
behavior of this model has not been considered. This is mainly due
to nonlinear nature of Smoluchowski coagulation equations used for
modeling aggregation. Hence, our main goal in this paper is to study
the long-term behavior of the broad class of flocculation models described
in (\ref{eq: agg and growth model}). For the remainder of this work
, we make the following assumptions
\begin{alignat*}{1}
(\mathbf{A}1)\qquad & g\in C^{1}(I)\qquad g(x)>0\:\text{ for }x\in I\\
(\mathbf{A}2)\qquad & k_{a}\in W^{1,\,\infty}(I\times I),\quad k_{a}(x,\,y)=k_{a}(y,\,x)\\
 & \text{ and }k_{a}(x,\,y)=0\,\,\,\text{ if }x+y\ge x_{1}\,,\\
(\mathbf{A}3)\qquad & \mu\in C(I)\qquad\text{and }\mu\ge0\text{ a.e. on }I\,,\\
(\mathbf{A}4)\qquad & q\in L^{\infty}(I)\qquad\text{and }q\ge0\text{ a.e. on }I\,,\\
(\mathbf{A}5)\qquad & k_{f}\in C(I)\qquad k_{f}(0)=0\text{ and }k_{f}\ge0\text{ a.e. on }I\,,\\
(\mathbf{A}6)\qquad & \Gamma(\cdot,\,y)\in W^{1,\,\infty}(I),\quad\Gamma(x;\,y)\ge0\text{ for }x\in(0,\,y];\\
 & \text{and }\Gamma(x;\,y)=0\text{ for }x\in(y,\,x_{1})\,.
\end{alignat*}

Assumption ($\mathbf{A}1$) states that the floc of any size has strictly
positive growth rate. This in turn implies that flocs can grow beyond
the maximal size $x_{1}$, i.e., the model ignores what happens beyond
the maximal size $x_{1}$ (as many authors in the literature have
done \citep{Farkas2010,AcklehFitzpatrick1997,Ackleh1997,Farkas2007}).
We also note that the Assumption $(\mathbf{A}1)$ generates biologically
unrealistic condition $g(0)>0$, i.e., the flocs of size zero also
have positive growth rate. However, this assumption is crucial for
our work, and thus we postpone the analysis of the case $g(0)=0$
for our future papers. Assumption $(\mathbf{A}2)$ states that for
the aggregates of size $x$ and $y$ the aggregation rate is zero
if the combined size of the aggregates is larger than the maximal
size. Lastly, Assumption $(\mathbf{A}3)$ on $\mu(x)$ enforces continuous
dependence of the removal on the size of a floc and ensures that every
floc is removed with a non-negative rate.

When the long-term behavior of biological populations is considered,
many populations converge to a stable time-independent state. Thus,
identifying conditions under which a population converges to a stationary
state is one of the most important applications of mathematical population
modeling. It is trivially true that a zero stationary solution exists,
but we are also interested in non-trivial stationary solutions of
the flocculation model. Hence, in Section \ref{sec:Existence-of-a postive}
we first show that under some suitable conditions on the model parameters
the flocculation equation has at least one non-trivial (non-zero and
non-negative) stationary solution. Once a stationary solution to a
model is shown to exist, the next natural question is whether it is
stable or unstable. When the associated evolution equation of a population
model is linear, many of stability properties can be deduced from
the spectral properties of this linear operator \citep{Diekmann1984c,Greiner1988a}.
However, almost no information about the operator can be deduced from
the spectrum of a nonlinear operator \citep{appell2004}. Moreover,
there is no general consensus among mathematicians on how to define
spectrum of a nonlinear operator. Thus, our stability analysis in
this work is based on the \emph{principle of linearized stability}
for nonlinear evolution equations \citep{Webb1985,Kato1995}. Hence,
in Section \ref{sec:Principle-of-linearized} we summarize the principle
of linearized stability and linearize the flocculation model around
its stationary solutions. In Section \ref{sec:Principle-of-linearized}
we also derive conditions for the regularity of the linearized flocculation
model. Next, in Sections \ref{sec:zero Linearized-stability-results}
and \ref{sec:Linearized-stability-(instabilit} we derive sufficiency
conditions for the linearized stability and instability of zero and
non-zero stationary solutions. In Section \ref{sec:Illustration-of-the results}
we illustrate our results with several examples. Finally, in Section
\ref{sec:Concluding-remarks}, we summarize and discuss the conclusions
of this work.

\section{\label{sec:Existence-of-a postive}Existence of a positive stationary
solution}

The flocculation model under our consideration (\ref{eq: agg and growth model}),
accounts for physical mechanisms such as growth, removal, fragmentation,
aggregation and renewal of microbial flocs. Thus, under some conditions,
which balance these mechanisms, one could reasonably expect that the
model possesses a non-trivial stationary solution. Hence, our main
goal in this section is to derive sufficient conditions for the model
terms such that the equation (\ref{eq: agg and growth model}) engenders
a positive stationary solution.

Recall that at a steady state we should have
\begin{equation}
p_{t}=0=\mathcal{F}[p]\,.\label{eq:steady state equation}
\end{equation}
By Assumption $(\mathbf{A}1)$, we know that $1/g\in C(I)$ and thus
we can define $p=f/g$ for some $f\in C(I)$. The substitution of
this $f$ into (\ref{eq:steady state equation}), integration between
$0$ and an arbitrary $x$, and rearrangment of the terms yields
\begin{alignat*}{1}
f(x) & =\int_{0}^{x_{1}}\frac{q(y)}{g(y)}f(y)\,dy-\int_{0}^{x}\frac{k_{f}(y)/2+\mu(y)}{g(y)}f(y)\,dy+\int_{0}^{x}\int_{z}^{x_{1}}\frac{\Gamma(z;\,y)k_{f}(y)}{g(y)}f(y)\,dy\,dz\\
 & \quad+\frac{1}{2}\int_{0}^{x}\int_{0}^{z}\frac{k_{a}(z-y,\,y)}{g(z-y)g(y)}f(z-y)f(y)\,dy\,dz-\int_{0}^{x}\frac{f(z)}{g(z)}\int_{0}^{x_{1}}\frac{k_{a}(z,\,y)}{g(y)}f(y)\,dy\,dz\,.
\end{alignat*}
We now define the operator $\Phi$ as
\begin{alignat}{1}
\Phi[f](x) & :=\int_{0}^{x_{1}}\frac{q(y)}{g(y)}f(y)\,dy-\int_{0}^{x}\frac{k_{f}(y)/2+\mu(y)}{g(y)}f(y)\,dy+\int_{0}^{x}\int_{z}^{x_{1}}\frac{\Gamma(z;\,y)k_{f}(y)}{g(y)}f(y)\,dy\,dz\nonumber \\
 & \quad+\frac{1}{2}\int_{0}^{x}\int_{0}^{z}\frac{k_{a}(z-y,\,y)}{g(z-y)g(y)}f(z-y)f(y)\,dy\,dz-\int_{0}^{x}\frac{f(z)}{g(z)}\int_{0}^{x_{1}}\frac{k_{a}(z,\,y)}{g(y)}f(y)\,dy\,dz\,.\label{eq:definition of phi}
\end{alignat}
and will use a fixed point theorem to prove the existence of a fixed
point $f$ of $\Phi$ . This in turn will allow us to claim that equation
(\ref{eq:steady state equation}) has at least one non-trivial positive
solution.

The use of fixed point theorems for showing existence of non-trivial
stationary solutions is not new in size-structured population modeling.
For example, fixed point theorems, based on Leray-Schauder degree
theory, have been used to find stationary solutions of linear Sinko-Streifer
type equations \citep{Farkas2010,Pruss1983}. Moreover, the Schauder
fixed point theorem has been used to establish the existence of steady
state solutions of nonlinear coagulation-fragmentation equations \citep{Laurencot2005}.
For our purposes we will use the following fixed point theorem, and
refer readers to \citep{Amann1976} for the full discussion of the
proof.\footnote{Hereafter, we refer to the following theorem as ``the fixed point
theorem''}
\begin{theorem}
Let $\mathcal{X}$ be a Banach space, $K\subset\mathcal{X}$ a closed
convex cone, $K_{r}=K\cap B_{r}(0),\,\Phi:\,K_{r}\to K$ continuous
such that $\Phi\left(K_{r}\right)$ is relatively compact. Assume
that 
\begin{enumerate}
\item $\Phi[x]\ne\lambda x$ for all $\left\Vert x\right\Vert =r$ and $\lambda>1$.
\item There exists a $\rho\in(0,\,r)$ and $k\in K\backslash\{0\}$ such
that 
\[
x-\Phi[x]\ne\lambda k\quad\text{for all }\left\Vert x\right\Vert =\rho\text{ and }\lambda>0\,.
\]

\end{enumerate}
Then $\Phi$ has at least one fixed point $x_{0}\in K$ such that
$\rho<\left\Vert x_{0}\right\Vert <r$.
\end{theorem}

Next, we show that the operator $\Phi$ defined in (\ref{eq:definition of phi})
satisfies the assumptions of the above theorem, which in turn implies
existence of a positive stationary solution of the operator $\mathcal{F}$.
Since we have been working on the space of absolutely integrable functions
on $I$, a natural candidate for the Banach space $\mathcal{X}$ would
be $L^{1}(I)$. However, to obtain sufficient regularity for stability
analysis of a stationary solution, we choose $\mathcal{X}=C(I)$ with
usual uniform norm $\left\Vert \cdot\right\Vert _{u}$ on $I$. We
also denote the usual essential supremum of a function by $\left\Vert \cdot\right\Vert _{\infty}$.
Since the positive cone in $C(I)$, denoted by $\left(C(I)\right)_{+}$,
is closed and convex, we choose $K$ to be $\left(C(I)\right)_{+}$.
Then $K_{r}=K\cap B_{r}(0)$, where $B_{r}(0)\subset\mathcal{X}$
is an open ball of radius $r$ and centered at zero, and $r$ has
yet to be chosen. We are now in a position to state the main result
of this section in the following theorem.
\begin{theorem}
\label{thm:positive steady state existence}Assume that the conditions
\[
0<q(x)+\frac{1}{2}k_{f}(x)-\mu(x),\tag{\textbf{C}1}
\]
and 
\[
\int_{x}^{x_{1}}\frac{1}{g(y)}\left(k_{f}(y)\int_{0}^{x}\Gamma(z;\,y)\,dz+q(y)\right)\,dy+\int_{0}^{x}\frac{1}{g(y)}\left(q(y)+\frac{1}{2}k_{f}(y)-\mu(y)\right)\,dy\le1\tag{\textbf{C}2}\,,
\]
hold true for all $x\in I$. Then the operator $\Phi$ defined in
(\ref{eq:definition of phi}) has at least one non-zero fixed point,
$f_{*}\in K$ satisfying 
\begin{equation}
0<\eta<\left\Vert f_{*}\right\Vert _{u}<r\label{eq:bounds for the fixed point}
\end{equation}
for some $\eta,\,r>0$. Moreover, the non-zero and non-negative function
\begin{equation}
p_{*}=\frac{f_{*}}{g}\in C(I)\label{eq:stationary solution}
\end{equation}
is the stationary solution of the flocculation model defined in (\ref{eq: agg and growth model}).\end{theorem}
\begin{proof}
For $f\in K_{r}$ we have 
\begin{alignat*}{1}
\Phi[p] & \ge\int_{0}^{x_{1}}\frac{q(y)}{g(y)}f(y)\,dy-\int_{0}^{x}\frac{k_{f}(y)/2+\mu(y)}{g(y)}f(y)\,dy\\
 & \quad+\int_{0}^{x}\int_{z}^{x}\frac{\Gamma(z;\,y)k_{f}(y)}{g(y)}f(y)\,dy\,dz-\int_{0}^{x}\frac{f(z)}{g(z)}\int_{0}^{x_{1}}\frac{k_{a}(z,\,y)}{g(y)}f(y)\,dy\,dz\\
 & \ge\int_{0}^{x}\frac{f(z)}{g(z)}\left[q(z)-\frac{1}{2}k_{f}(z)-\mu(z)-\int_{0}^{x_{1}-z}\frac{k_{a}(z,\,y)}{g(y)}f(y)\,dy\right]\,dz+\int_{0}^{x}\frac{k_{f}(y)f(y)}{g(y)}\underbrace{\int_{0}^{y}\Gamma(z;\,y)\,dz}_{=1}\,dy\\
 & \ge\int_{0}^{x}\frac{f(z)}{g(z)}\left[q(z)-\frac{1}{2}k_{f}(z)-\mu(z)-r\cdot\left\Vert k_{a}\right\Vert _{\infty}\left\Vert \frac{1}{g}\right\Vert _{1}\right]\,dz\,,
\end{alignat*}
where $\left\Vert \cdot\right\Vert _{1}$ represents the usual $L^{1}$
norm on $I$. The first condition of the theorem $(\mathbf{C}1)$
guarantees that 
\[
q(z)+\frac{1}{2}k_{f}(z)-\mu(z)>0\text{ for all }z\in I\,,
\]
so we can choose $r$ in $K_{r}$ sufficiently small such that $\Phi[f]\ge0$,
i.e., $\Phi\,:\,K_{r}\to K$. On the other hand, using the assumptions
$(\mathbf{A}1)$-$(\mathbf{A}6)$, it is straightforward to show that
$\Phi(K_{r})\subset C^{1}(I)$.\footnote{Recall that continuous differentiability implies uniform continuity,
and thus equicontinuity} This in turn, from Arzelà-Ascoli theorem, implies that the operator
$\Phi$ has relatively compact image. 

Next we prove the second assumption of the theorem. For the sake of
contradiction, suppose that there exist $f\in K_{r}$ with $\left\Vert f\right\Vert _{u}=r$
and $\lambda>1$ such that
\[
\Phi[f]=\lambda f\,.
\]
Then it follows that 
\begin{alignat*}{1}
\lambda f(x) & \le\int_{0}^{x_{1}}\frac{q(y)}{g(y)}f(y)\,dy-\int_{0}^{x}\frac{k_{f}(y)/2+\mu(y)}{g(y)}f(y)\,dy+\int_{0}^{x_{1}}\int_{z}^{x_{1}}\frac{\Gamma(z;\,y)k_{f}(y)}{g(y)}f(y)\,dy\,dz\\
 & \quad-\int_{x}^{x_{1}}\int_{z}^{x_{1}}\frac{\Gamma(z;\,y)k_{f}(y)}{g(y)}f(y)\,dy\,dz+\frac{1}{2}\int_{0}^{x}\int_{0}^{z}\frac{k_{a}(z-y,\,y)}{g(z-y)g(y)}f(z-y)f(y)\,dy\,dz\\
 & \le\int_{x}^{x_{1}}\frac{1}{g(y)}\left(k_{f}(y)\int_{0}^{x}\Gamma(z;\,y)\,dz+q(y)\right)f(y)\,dy\\
 & \quad+\int_{0}^{x}\frac{1}{g(y)}\left(q(y)+\frac{1}{2}k_{f}(y)-\mu(y)\right)f(y)\,dy+\left\Vert f\right\Vert _{u}^{2}\cdot\left\Vert k_{a}\right\Vert _{\infty}\cdot\left\Vert \frac{1}{g}\right\Vert _{1}^{2}\\
 & \le\left\Vert f\right\Vert _{u}\left[\int_{x}^{x_{1}}\frac{1}{g(y)}\left(k_{f}(y)\int_{0}^{x}\Gamma(z;\,y)\,dz+q(y)\right)\,dy\right.\\
 & \quad\left.+\int_{0}^{x}\frac{1}{g(y)}\left(q(y)+\frac{1}{2}k_{f}(y)-\mu(y)\right)\,dy+r\cdot\left\Vert k_{a}\right\Vert _{\infty}\cdot\left\Vert \frac{1}{g}\right\Vert _{1}^{2}\right]
\end{alignat*}
which yields that 
\begin{alignat*}{1}
\lambda & \le\sup\left\{ \int_{x}^{x_{1}}\frac{1}{g(y)}\left(k_{f}(y)\int_{0}^{x}\Gamma(z;\,y)\,dz+q(y)\right)\,dy+\int_{0}^{x}\frac{1}{g(y)}\left(q(y)+\frac{1}{2}k_{f}(y)-\mu(y)\right)\,dy\right\} \\
 & \quad+r\left\Vert k_{a}\right\Vert _{\infty}\left\Vert \frac{1}{g}\right\Vert _{1}^{2}\,.
\end{alignat*}
From the second condition of the theorem $(\mathbf{C}2)$ it follows
that
\[
\sup\left\{ \int_{x}^{x_{1}}\frac{1}{g(y)}\left(k_{f}(y)\int_{0}^{x}\Gamma(z;\,y)\,dz+q(y)\right)\,dy+\int_{0}^{x}\frac{1}{g(y)}\left(q(y)+\frac{1}{2}k_{f}(y)-\mu(y)\right)\,dy\right\} \le1\,.
\]
Then, we can choose $r$ sufficiently small such that it contradicts
the first assumption of the fixed point theorem, $\lambda>1$.

Next we will derive conditions for the second condition of the fixed
point theorem. For the sake of contradiction, let us choose $k\equiv1\in K\backslash\{0\}$
and assume that there exists $f\in K_{r}$ with $\left\Vert f\right\Vert _{u}=\eta<r$
and $\lambda>0$ such that
\[
f-\Phi[f]=\lambda k=\lambda\,.
\]
This equation in turn can be written as 
\begin{alignat}{1}
\lambda k & =\lambda=f(x)-\Phi[f](x)\nonumber \\
 & \le\left\Vert f\right\Vert _{u}-\int_{0}^{x_{1}}\frac{q(y)}{g(y)}f(y)\,dy+\int_{0}^{x}\frac{k_{f}(y)/2+\mu(y)}{g(y)}f(y)\,dy\nonumber \\
 & \quad-\int_{0}^{x}\int_{z}^{x}\frac{\Gamma(z;\,y)k_{f}(y)}{g(y)}f(y)\,dy\,dz+\int_{0}^{x_{1}}\frac{f(z)}{g(z)}\int_{0}^{x_{1}}\frac{k_{a}(z,\,y)}{g(y)}f(y)\,dy\,dz\nonumber \\
 & \le\eta-\int_{0}^{x_{1}}\frac{q(y)}{g(y)}f(y)\,dy+\int_{0}^{x}\frac{\mu(y)-k_{f}(y)/2}{g(y)}f(y)\,dy+\eta^{2}\cdot\left\Vert k_{a}\right\Vert _{\infty}\cdot\left\Vert \frac{1}{g}\right\Vert _{1}^{2}\nonumber \\
 & \le\int_{0}^{x}\frac{\mu(y)-k_{f}(y)/2-q(y)}{g(y)}f(y)\,dy+\eta^{2}\cdot\left\Vert k_{a}\right\Vert _{\infty}\cdot\left\Vert \frac{1}{g}\right\Vert _{1}^{2}+\eta\,,\label{eq:negative lambda}
\end{alignat}
which should hold for all $x\in I$. Thus, provided that the condition
$(\mathbf{C}1)$ holds, we can choose $\eta\in\left(0,\,r\right)$
sufficiently small such that we get a contradiction to the second
assumption of the fixed point theorem, $\lambda>0$. Hence, the fixed
point theorem guarantees the existence of a positive fixed point of
$\Phi$ satisfying the bounds (\ref{eq:bounds for the fixed point}).

Therefore, the function $p_{*}=f_{*}/g$ is a stationary solution
of the flocculation equations (\ref{eq: agg and growth model}). Moreover,
from the assumption ($\mathbf{A}1$) and the continuity of the fixed
point $f_{*}$ it follows that $p_{*}$ is non-zero, non-negative
and continuous on $I$.
\end{proof}

\section{\label{sec:Principle-of-linearized}Principle of linearized stability
and regularity properties of the linearized semigroup}

In this section we summarize the principle of linearized stability
as it applies to semigroups in general and our flocculation equation
in particular.

For a given autonomous ordinary differential equation,
\[
\dot{u}=f(u)\,,
\]
the method for determining the local asymptotic behavior of a stationary
solution $u_{*}$, $f(u_{*})=0$, by the eigenvalues of the Jacobian
$\mathbf{J}_{f}(u_{*})$ is quite well-known. In semigroup theory
this method is known as the \emph{principle of linearized stability}
and was developed in the context of semilinear partial differential
equations in \citep{Henry1981,Smoller1983,Webb1985}. Later, Kato
\citep{Kato1995} extended this principle to a broader range of nonlinear
evolution equations. Before presenting the principle of linearized
stability we introduce some terminology, which can be found in many
functional analysis books (see \citep{Belleni-Morante1998a} for instance). 

The \emph{growth bound} $\omega_{0}(\mathbf{A})$ of a strongly continuous
semigroup $\left(S(t)\right)_{t\ge0}$ with an infinitesimal generator
$\mathbf{A}$ is defined as
\[
\omega_{0}(\mathbf{A}):=\inf\left\{ \omega\in\mathbb{R}\,:\,\begin{array}{c}
\exists M_{\omega}\ge1\text{ such that}\\
\left\Vert S(t)\right\Vert \le M_{\omega}e^{\omega t}\text{ for all }t\ge0
\end{array}\right\} \,.
\]
The operator $D\mathbf{A}(f)$ denotes the Fréchet derivative of an
operator $\mathbf{A}$ evaluated at $f$, which is defined as 
\[
D\mathbf{A}(u)h=\mathbf{A}[u+h]-\mathcal{\mathbf{A}}[u]+o(h),\qquad\forall u\in\mathcal{D}(\mathbf{A})\,,
\]
where $o$ is little-o operator satisfying $\left\Vert o(h)\right\Vert \le b(r)\left\Vert h\right\Vert $
with increasing continuous function\textbf{ $b\,:\,[0,\,\infty)\to[0,\,\infty),\,b(0)=0$}. 

The\emph{ discrete spectrum} $\sigma_{D}(\mathbf{A})$ of an arbitrary
operator $\mathbf{A}$ on a Banach space $X$, is the subset of the
point spectrum of $\mathbf{A}$, 
\[
\sigma_{p}(\mathbf{A})=\left\{ \lambda\in\mathbb{C}\,|\,\exists\phi\ne0\in X\text{ s.t. }\mathbf{A}\phi=\lambda\phi\right\} \,,
\]
such that $\lambda\in\sigma_{D}(\mathbf{A})$ is an isolated eigenvalue
of finite multiplicity, i.e., the dimension of the set 
\[
\left\{ \psi\in X\,:\,\mathbf{A}\psi=\lambda\psi\right\} 
\]
is finite and nonzero. Let $\left(T(t)\right)_{t\ge0}$ be a $C_{0}$
semigroup on the Banach space $X$ with its infintesimal generator
$\mathbf{A}$. Then the limit $\omega_{1}(\mathcal{\mathbf{A}})=\lim_{t\to\infty}t^{-1}\log\left(\alpha[T(t)]\right)$
is well-defined and called the \emph{$\alpha$-growth bound} of $\left(T(t)\right)_{t\ge0}$.
The function $\alpha[T(t)]$ is a measure of non-compactness of the
semigroup $T(t)$ as defined as in \citep{Kuratowski}. This measure
associates non-negative numbers to operators (or sets), which tells
how close an operator (or a set) is to a compact operator (or set).
For example, for a bounded set $M$ in a Banach space, $\alpha[M]=0$
implies that $\overline{M}$ (closure of $M$) is a compact set. Analogously,
for a semigroup $\left(T(t)\right)_{t\ge0}$, $\alpha[T(t)]=0$ indicates
that the semigroup is eventually compact.

With the above definitions, we are now ready to present the principle
of linearized stability in the form of the following proposition (see
\citep{Webb1985} for the complete discussion of the proof of the
following proposition).
\begin{proposition}
\label{prop:linearized stability} Define the nonlinear operator $\mathcal{N}\,:\,\mathcal{D}(\mathcal{F})\subset L^{1}(I)\to L^{1}(I)$
and let $f_{*}\in\mathcal{D}(\mathcal{N})$ be a stationary solution
of (\ref{eq: agg and growth model}), i.e., $\mathcal{N}[f_{*}]=0$.
If $\mathcal{N}$ is continuously Fréchet differentiable on $L^{1}(I)$
and the linearized operator $\mathcal{L}=D\mathcal{N}(f_{*})$ is
the infinitesimal generator of a $C_{0}$-semigroup $T(t)$, then
the following statements hold:
\begin{enumerate}
\item If $\omega_{0}\left(\mathcal{L}\right)<0$, then $f_{*}$ is locally
asymptotically stable in the following sense: There exists $\eta,\,C\ge1,$
and $\alpha>0$ such that if $\left\Vert f-f_{*}\right\Vert <\eta$,
then a unique mild solution $T(t)f$, satisfies$\left\Vert T(t)f-f_{*}\right\Vert \le Ce^{-\alpha t}\left\Vert f-f_{*}\right\Vert $
for all $t\ge0$. 
\item If there exists $\lambda_{0}\in\sigma(\mathcal{L})$ such that $\text{Re}\,\lambda>0$
and 
\begin{equation}
\max\left\{ \omega_{1}(\mathcal{L}),\,\sup_{\lambda\in\sigma_{D}(\mathcal{L})\backslash\{\lambda_{0}\}}\text{Re}\,\lambda\right\} <\text{Re}\,\lambda_{0}\,,\label{eq: instability condition}
\end{equation}
then $f_{*}$ is an unstable equilibrium in the sense that there exists
$\varepsilon>0$ and sequence $\left\{ f_{n}\right\} $ in $X$ such
that $f_{n}\to f_{*}$ and $\left\Vert T(n)f_{n}-f_{*}\right\Vert \ge\varepsilon$
for $n=1,2,\dots$ .
\end{enumerate}
\end{proposition}
Having the explicit statement of the principle of linearized stability
in hand, we now show that the nonlinear operator $\mathcal{F}$ defined
in (\ref{eq:Breakage}) satisfies all the conditions of Proposition
\ref{prop:linearized stability}. Towards this end, we first establish
the elementary assumption of Proposition \ref{prop:linearized stability}
in the following lemma.
\begin{lemma}
The nonlinear operator $\mathcal{F}$ defined in (\ref{eq:Breakage})
is continuously Fréchet differentiable on $L^{1}(I)$.\end{lemma}
\begin{proof}
The Fréchet derivative of the nonlinear operator $\mathcal{F}$ is
given explicitly as

\begin{alignat*}{1}
D\mathcal{F}(\phi)[h(x)] & =-\partial_{x}[gh](x)-\left(\mu(x)+\frac{1}{2}k_{f}(x)\right)h(x)+\int_{x}^{x_{1}}\Gamma(x;\,y)k_{f}(y)h(y)\,dy\\
 & \quad+\frac{1}{2}\int_{0}^{x}k_{a}(x-y,\,y)\left[\phi(y)h(x-y)+h(y)\phi(x-y)\right]dy\\
 & \quad-h(x)\int_{0}^{x_{1}-x}k_{a}(x,\,y)\phi(y)dy-\phi(x)\int_{0}^{x_{1}-x}k_{a}(x,\,y)h(y)dy\,.
\end{alignat*}
For the arbitrary functions $u_{1},\,u_{2}\in L^{1}(I)$ we have 
\begin{alignat*}{1}
\left|D\mathcal{F}(u_{1})h(x)-D\mathcal{F}(u_{2})h(x)\right| & \le\frac{1}{2}\left\Vert k_{a}\right\Vert _{\infty}\int_{0}^{x}|u_{1}(y)-u_{2}(y)||h(x-y)|\,dy\\
 & \quad+\frac{1}{2}\left\Vert k_{a}\right\Vert _{\infty}\int_{0}^{x}|h(y)||u_{1}(x-y)-u_{2}(x-y)|\,dy\\
 & \quad+|h(x)|\left\Vert k_{a}\right\Vert _{\infty}\int_{0}^{x_{1}}|u_{1}(y)-u_{2}(y)|dy\\
 & \quad+|u_{1}(x)-u_{2}(x)|\left\Vert k_{a}\right\Vert _{\infty}\int_{0}^{x_{1}}|h(y)|dy
\end{alignat*}
Consequently, taking the integral of both sides with respect to $x$
and an application of Young's inequality for convolutions (see \citep[Theorem 2.24]{Adams2003})
to the first two integrals yields
\[
\left\Vert D\mathcal{F}(u_{1})h(x)-D\mathcal{F}(u_{2})h(x)\right\Vert _{1}\le\left\Vert k_{a}\right\Vert _{\infty}\left\Vert u_{1}-u_{2}\right\Vert _{1}\left\Vert h\right\Vert _{1}+\left\Vert k_{a}\right\Vert _{\infty}\left\Vert u_{1}-u_{2}\right\Vert _{1}\left\Vert h\right\Vert _{1}+\left\Vert u_{1}-u_{2}\right\Vert _{1}\left\Vert k_{a}\right\Vert _{\infty}\left\Vert h\right\Vert _{1}
\]
for all $h\in L^{1}(I)$. Then it follows that 
\[
\left\Vert D\mathcal{F}(u_{1})-D\mathcal{F}(u_{2})\right\Vert _{1}\le3\left\Vert k_{a}\right\Vert _{\infty}\left\Vert u_{1}-u_{2}\right\Vert _{1}\,,
\]
which in turn implies that the nonlinear operator $\mathcal{F}$ is
continuously Fréchet differentiable on $L^{1}(I)$.
\end{proof}

In the previous section we have shown that the nonlinear operator
$\mathcal{F}$ (\ref{eq:Breakage}) has at least one non-trivial stationary
solution, $p_{*}$ (in addition to trivial zero stationary solution).
To derive stability results for this stationary solutions we first
linearize the equation (\ref{eq: agg and growth model}) around $p_{*}$.
A simple calculation yields that the Fréchet derivative of the nonlinear
operator $\mathcal{F}$ evaluated at a stationary solution $p_{*}$
(see Theorem \ref{thm:positive steady state existence}) is given
explicitly by

\begin{alignat}{1}
\mathcal{L}[h](x)=D\mathcal{F}(p_{*})[h](x) & =-\partial_{x}(g(x)h(x))-A(x)h(x)+\int_{x}^{x_{1}}\Gamma(x;\,y)k_{f}(y)h(y)\,dy\nonumber \\
 & \quad-\int_{0}^{x_{1}-x}E(x,\,y)h(y)\,dy+\int_{0}^{x}E(x-y,\,y)h(y)\,dy\,,\label{eq:linearized operator L}
\end{alignat}
where
\[
E(x,\,y)=k_{a}(x,\,y)p_{*}(x)
\]
and 
\[
A(x)=\frac{1}{2}k_{f}(x)+\mu(x)+\int_{0}^{x_{1}-x}E(y,\,x)\,dy\,.
\]

We first prove that the linear operator $\mathcal{L}$ is an infinitesimal
generator of a strongly continuous semigroup $\mathcal{\mathscr{T}}=\left(T(t)\right)_{t\ge0}$.
Consequently, we will prove two regularity results for the semigroup
$\mathscr{T}$ , which will prove useful in the spectral analysis
of the operator $\mathcal{L}$. Particularly, we will show that under
some conditions on the model ingredients the semigroup $\mathscr{T}$
is positive and eventually compact. The main implication of eventual
compactness is that the Spectral Mapping Theorem holds (see \citep{Engel2000})
for the semigroup $\mathscr{T}$,
\[
\sigma\left(T(t)\right)\backslash\{0\}=\exp\left(t\sigma(\mathcal{L})\right),\quad t\ge0\,.
\]
 Consequently, we will use the positivity of the semigroup $\mathscr{T}$
in Section \ref{sub:Linearized-stability}, where we employ the positive
perturbation method introduced in \citep{Farkas2010}.
\begin{lemma}
\label{lem:operator L generates C0}If we define the domain of the
linearized operator $\mathcal{L}$ as 
\begin{equation}
\mathcal{D}(\mathcal{L})=\left\{ \phi\in L^{1}(I)\,|\,(g\phi)'\in L^{1}(I),\,(g\phi)(0)=\mathcal{K}[\phi]\right\} \,,\label{eq:domain of L}
\end{equation}
then the operator \emph{$\mathcal{L}$} generates a $C_{0}$ semigroup
on $\mathcal{D}(\mathcal{L})$.\end{lemma}
\begin{proof}
The linear operator $\mathcal{L}$ can be written as the sum of an
unbounded operator 
\begin{equation}
\mathcal{L}_{1}[h](x)=-\partial_{x}(g(x)h(x))-A(x)h(x)\label{eq:operator L1}
\end{equation}
and bounded operators
\begin{equation}
\mathcal{L}_{2}[h](x)=\int_{x}^{x_{1}}\Gamma(x;\,y)k_{f}(y)h(y)\,dy-\int_{0}^{x_{1}-x}E(x,\,y)h(y)\,dy,\qquad\mathcal{L}_{3}[h](x)=\int_{0}^{x}E(x-y,\,y)h(y)\,dy\,.\label{eq:operators L2 and L3}
\end{equation}
From the fact that $g(x),\,A(x)\in C(I)$ and from the Lemma 2.4 of
\citep{Lamb2009} it follows that $\mathcal{L}_{1}$ generates a $C_{0}$
semigroup on $\mathcal{D}(\mathcal{L})$. Consequently, the bounded
perturbation theorem of \citep[\S 3, Theorem 1.1]{Pazy1992} yields
that the operator $\mathcal{L}$ is also an infinitesimal generator
of a $C_{0}$ semigroup.\end{proof}
\begin{lemma}
\label{lem:L2 is compact}For a given stationary solution $p_{*}\in C(I)$
the operators $\mathcal{L}_{2}\,:\,\mathcal{D}(\mathcal{L})\to L^{1}(I)$
and $\mathcal{L}_{3}\,:\,\mathcal{D}(\mathcal{L})\to L^{1}(I)$ defined
in (\ref{eq:operators L2 and L3}) are compact operators. \end{lemma}
\begin{proof}
We first prove that the operator $\mathcal{L}_{2}$ is compact. Then
compactness of the operator $\mathcal{L}_{3}$ follows from analogous
arguments. Let us denote a unit ball centered at zero in $L^{1}(I)$
by $B=\left\{ \phi\in L^{1}(I)\,|\,\left\Vert \phi\right\Vert _{1}\le1\right\} $.
Recall that an operator is compact if it maps a unit ball into a relatively
compact set. Consequently, observe that the assumptions ($\mathbf{A}2$)
and $(\mathbf{A}6)$ together imply that the operator 
\begin{alignat*}{1}
\partial_{x}\mathcal{L}_{2}[h](x) & =k_{a}(x,\,x_{1}-x)p_{*}(x)h(x_{1}-x)-\int_{0}^{x_{1}-x}\partial_{x}(k_{a}(x,\,y)p_{*}(x))h(y)\,dy\\
 & \quad+\int_{x}^{x_{1}}\partial_{x}\Gamma(x;\,y)k_{f}(y)h(y)\,dy-\Gamma(x;\,x)k_{f}(x)h(x)
\end{alignat*}
is also bounded. Hence $\mathcal{L}_{2}[B]\subset W^{1,1}(I)$ and
from the Rellich-Kondrachov embedding theorem (see \emph{\citep[Theorem 6.3]{Adams2003}}
for a statement of the theorem) it follows that the set $\mathcal{L}_{2}[B]$
is relatively compact.\end{proof}
\begin{lemma}
\label{lem:eventually compact} The operator $\mathcal{L}$ defined
in (\ref{eq:linearized operator L}) generates an eventually compact
$C_{0}$ semigroup. And thus, the spectrum of the operator $\mathcal{L}$
consists of isolated eigenvalues of a finite multiplicity only, i.e.,
$\sigma(\mathcal{L})=\sigma_{D}(\mathcal{L})$.\end{lemma}
\begin{proof}
The operator $\mathcal{L}_{1}$ defined in (\ref{eq:operator L1})
is well-known operator in size-structured dynamics literature. If
$g\in C^{1}(I)$ and $A\in C(I)$, then in \citet[Theorem 3.1]{Farkas2007}
it has been shown that the $C_{0}$ semigroup generated by the operator
$\mathcal{L}_{1}$ is compact for $t>2\int_{0}^{x_{1}}\frac{1}{g(y)}\,dy$.
The condition $g\in C^{1}(I)$ follows from our main assumption ($\mathbf{A}1$),
and continuity of the function
\[
A(x)=\frac{1}{2}k_{f}(x)+\mu(x)+\int_{0}^{x_{1}-x}k_{a}(x,\,y)p_{*}(y)\,dy
\]
follows from the assumptions $(\mathbf{A}1$)- $(\mathbf{A}6$). Thus
the semigroup generated by $\mathcal{L}_{1}$ is eventually compact.
Conversely, in Lemma (\ref{lem:L2 is compact}) we have shown that
the operators $\mathcal{L}_{2}$ and $\mathcal{L}_{3}$ are compact.
Hence, the $C_{0}$ semigroup generated by the operator $\mathcal{L}=\mathcal{L}_{1}+\mathcal{L}_{2}+\mathcal{L}_{3}$
is also compact for $t>2\int_{0}^{x_{1}}\frac{1}{g(y)}\,dy$.

Therefore, the eventual compactness of the semigroup $\mathscr{T}$
(generated by $\mathcal{L}$) combined with Theorem 3.3 of \citep[\S 2.3]{Pazy1992}
and Corollary 1.19 of \citep[\S 4]{Engel2000} together imply that
the spectrum of $\mathcal{L}$ consists of isolated eigenvalues of
finite multiplicity.\end{proof}
\begin{lemma}
\label{lem:positve semigroup}For a given steady state solution $p_{*}$
let us choose the functions $k_{a}$, $k_{f}$ and $\Gamma$ such
that 
\begin{equation}
\partial_{x}\left(k_{a}(x,\,y)p_{*}(x)\right)\le0\text{ for all }x\in I\text{ and }y\in(0,\,x)\label{eq:positivity condition 1}
\end{equation}
and 
\begin{equation}
\Gamma(x;\,y)k_{f}(y)\ge k_{a}(x,\,y)p_{*}(x)\text{ for all }x\in I\text{ and }y\in[x,\,x_{1})\label{eq:positivity condition 2}
\end{equation}
 Then the operator $\mathcal{L}$ generates a positive $C_{0}$ semigroup. \end{lemma}
\begin{proof}
In \citet[Theorem 3.3]{Farkas2007} it has been shown that the operator
$\mathcal{L}_{1}\,:\,\mathcal{D}(\mathcal{L})\to L^{1}(I)$ generates
a positive $C_{0}$ semigroup under the main assumptions ($\mathbf{A}1$)-($\mathbf{A}6$).
On the other hand, from the conditions (\ref{eq:positivity condition 1})
and (\ref{eq:positivity condition 2}) it follows that 
\begin{alignat*}{1}
\mathcal{L}_{2}[h](x)+\mathcal{L}_{3}[h](x) & \ge\int_{0}^{x}k_{a}(x-y,\,y)p_{*}(x-y)h(y)\,dy+\int_{x}^{x_{1}}\Gamma(x;\,y)k_{f}(y)h(y)\,dy\\
 & \quad-\int_{0}^{x_{1}}k_{a}(x,\,y)p_{*}(x)h(y)\,dy\\
 & \ge\int_{0}^{x}\left[k_{a}(x-y,\,y)p_{*}(x-y)-k_{a}(x,\,y)p_{*}(x)\right]h(y)\,dy\\
 & \quad+\int_{x}^{x_{1}}\left[\Gamma(x;\,y)k_{f}(y)-k_{a}(x,\,y)p_{*}(x)\right]h(y)\,dy\ge0\,,
\end{alignat*}
which in turn ensures that the operator $\mathcal{L}_{2}+\mathcal{L}_{3}$
is a positive operator. Since the positivity of a semigroup is invariant
under a bounded and positive perturbation of its generator (see \citep[\S 6, Corollary 1.11]{Engel2000}),
the result follows immediately.
\end{proof}
\begin{rem}\label{rem:growth bound in spectrum}\end{rem}Lemma \ref{lem:positve semigroup}
has very important consequence. Specifically, if the positivity conditions
(\ref{eq:positivity condition 1}) and (\ref{eq:positivity condition 2})
hold and the spectral bound $s(\mathcal{L})=\sup\left\{ Re\,\lambda\,|\,\lambda\in\sigma(\mathcal{L})\right\} $
is not equal to $-\infty$, then $s(\mathcal{L})$ belongs to the
spectrum $\sigma(\mathcal{L})$ \citet[\S Theorem 1.10]{Engel2000}.
Moreover, the positivity and eventual compactness of the semigroup
$\mathscr{T}$ together imply that the spectral bound $s(\mathcal{L})$
is one of the eigenvalues of $\mathcal{L}$ with finite multiplicity.

\section{\label{sec:zero Linearized-stability-results}Linearized stability
and instability criteria for the zero stationary solution}

In this section we will derive linearized stability results for the
zero stationary solution of the flocculation equation. In contrast
to non-trivial stationary solutions, zero stationary solution always
exists (provided that the well-posedness assumptions ($\mathbf{A}1$)-($\mathbf{A}6$)
hold true). As we have discussed in Section \ref{sec:Principle-of-linearized}
the stability of the steady states depends on the spectral properties
of the linear operator $\mathcal{L}$ defined in (\ref{eq:linearized operator L}).
We define the operator $\mathcal{M}$ as the linear operator $\mathcal{L}$
evaluated at the trivial stationary solution, $p_{*}\equiv0$
\begin{equation}
\mathcal{M}[h](x)=-\partial_{x}[gh](x)-\left(\mu(x)+\frac{1}{2}k_{f}(x)\right)h(x)+\int_{x}^{x_{1}}\Gamma(x;\,y)k_{f}(y)h(y)\,dy\,.\label{eq:operator M}
\end{equation}

The assumptions ($\mathbf{A}1$)-($\mathbf{A}6$) also ensure that
the regularity conditions of Section \ref{sec:Principle-of-linearized}
are all satisfied. Hence, the operator $\mathcal{M}$ generates a
positive, eventually compact and strongly continuous semigroup. By
Remark \ref{rem:growth bound in spectrum} we know that the spectral
bound $s(\mathcal{M})$ of the operator $\mathcal{M}$ is a dominant
eigenvalue of $\mathcal{M}$ with finite multiplicity. Then, by the
principle of linearized stability (Proposition \ref{prop:linearized stability}),
the stability of the zero stationary solution depends on the sign
of this dominant eigenvalue. Thus, in the subsequent two subsections
we derive conditions which guarantee positivity and negativity of
the spectral bound $s(\mathcal{M})$, respectively. For a more thorough
discussion of the stability of zero stationary solution we refer readers
to \citep{Mirzaev2015a}.

\subsection{\label{sub:Instability-of-the trivial }Instability of the trivial
stationary solution}

The operator $\mathcal{M}$ can be written as the sum of an unbounded
operator 
\[
\mathcal{M}_{1}[h](x)=-\partial_{x}[gh](x)-\left(\mu(x)+\frac{1}{2}k_{f}(x)\right)h(x)
\]
and a bounded operator 
\[
\mathcal{M}_{2}[h](x)=\int_{x}^{x_{1}}\Gamma(x;\,y)k_{f}(y)h(y)\,dy\,.
\]
In \citep{Farkas2007}, authors have shown that the operator $\mathcal{M}_{1}$
generates a positive, eventually compact semigroup. Moreover, authors
have shown that the spectral bound of \emph{$\mathcal{M}_{1}$} is
positive if 
\begin{equation}
\int_{0}^{x_{1}}\frac{q(x)}{g(x)}\exp\left(-\int_{0}^{x}\frac{\mu(s)+\frac{1}{2}k_{f}(s)}{g(s)}\,ds\right)\,dx>1\,.\label{eq:zero instability condition}
\end{equation}
On the other hand, we note that $\mathcal{M}_{2}$ is a positive operator.
Then, Corollary 1.11 of \citep[\S 6]{Engel2000} yields that the operator
$\mathcal{M}=\mathcal{M}_{1}+\mathcal{M}_{2}$ also generates a positive,
eventually compact semigroup. Furthermore, the following inequality
holds for spectral bound of $\mathcal{M}_{1}$ and $\mathcal{M}$,
\begin{equation}
s(\mathcal{M}_{1})\le s(\mathcal{M}_{1}+\mathcal{M}_{2})=s(\mathcal{M})\,.\label{eq:ineq for spectral bounds}
\end{equation}
Consequently, this implies that the operator $\mathcal{M}$ also has
a positive spectral bound provided that the condition (\ref{eq:zero instability condition})
is satisfied. At this point, in Proposition \ref{prop:linearized stability},
choosing $\lambda_{0}$ equal to the eigenvalue of $\mathcal{M}$
corresponding to $s(\mathcal{M})$ and using Lemma \ref{lem:eventually compact}
yields 
\[
\max\left\{ \omega_{1}(\mathcal{M}),\,\sup_{\lambda\in\sigma_{D}(\mathcal{M})\backslash\{\lambda_{0}\}}\text{Re}\,\lambda\right\} =\sup_{\lambda\in\sigma_{D}(\mathcal{M})\backslash\{\lambda_{0}\}}\text{Re}\,\lambda<\text{Re}\,\lambda_{0}\,.
\]
Then, the operator $\mathcal{M}$ satisfies all the conditions of
Proposition \ref{prop:linearized stability} and thus results of this
section can be summarized in the form of the following condition.

\begin{cond}Assume that the assumptions ($\mathbf{A}1$)-($\mathbf{A}6$)
hold true. Moreover, assume that 
\[
\int_{0}^{x_{1}}\frac{q(x)}{g(x)}\exp\left(-\int_{0}^{x}\frac{\mu(s)+\frac{1}{2}k_{f}(s)}{g(s)}\,ds\right)\,dx>1\,,
\]
then the zero stationary solution of the flocculation equation is
unstable.\end{cond}

\subsection{Stability of the trivial stationary solution}

In this section we will prove that under certain condition on model
parameters we can ensure that the spectral bound of $\mathcal{M}$
is strictly negative. Since the positivity arguments that we used
in the previous section cannot guarantee negativity of $s(\mathcal{M})$,
we use a direct approach to prove that growth bound of $\mathcal{M}$
is strictly negative, $\omega_{0}(\mathcal{M})<0$. To achieve our
goal we use the following version of the well-known Lumer-Philips
theorem (see for instance \citep[\S 2, Corollary 3.6]{Engel2000}
and \citep[Theorem 2.22]{Belleni-Morante1998a}).
\begin{theorem}
(Lumer-Philips) Let a linear operator $\mathcal{A}$ on a Banach space
($\mathcal{X}$, $\left\Vert \cdot\right\Vert $) the following are
equivalent:
\begin{enumerate}
\item $\mathcal{A}$ is closed, densely defined. Furthermore, $\mathcal{A}-\lambda I$
is surjective for some $\lambda>0$ (and hence for all $\lambda>0$)
and there exists a real number $\omega$ such that $\mathcal{A}-\omega I$
is dissipative, i.e.,
\[
\left\Vert f-\lambda(\mathcal{A}-\omega I)f\right\Vert \ge\left\Vert f\right\Vert \quad\text{for all }\lambda>0\text{ and }f\in\mathcal{D}(\mathcal{A})\,.
\]

\item Then, $\left(\mathcal{A},\,\mathcal{D}(\mathcal{A})\right)$ generates
a strongly continuous quasicontractive semigroup $\left(T(t)\right)_{t\ge0}$
satisfying 
\[
\left\Vert T(t)\right\Vert \le e^{\omega t}\quad\text{for }t\ge0\,.
\]

\end{enumerate}
\end{theorem}
In the following lemma, we show the operator satisfies the first part
of the Lumer-Philips theorem. Particularly, we establish that there
exist a strictly negative real number $w<0$ such that the operator
$\mathcal{M}-\omega I$ is dissipative.
\begin{lemma}
Assume that the assumptions ($\mathbf{A}1$)-($\mathbf{A}6$) hold
true. Then the linear operator $\mathcal{M}$ defined in (\ref{eq:operator M})
is closed, densely defined operator on the Banach space $L^{1}(I)$,
and for sufficiently large $\lambda>0$ the operator $\mathcal{M}-\lambda I\,:\,\mathcal{D}(\mathcal{M})\mapsto L^{1}(I)$
is surjective. Furthermore, if
\begin{equation}
\mu(x)-q(x)-\frac{1}{2}k_{f}(x)>0\label{eq:condition for stability}
\end{equation}
for all $x\in I$, then there exists $\alpha>0$ such that the semigroup
$\left(T(t)\right)_{t\ge0}$ generated by $\mathcal{M}$ satisfies
the estimate 
\[
\left\Vert T(t)\right\Vert _{1}\le e^{-\alpha t}\text{ for all }t\ge0\,.
\]
\end{lemma}
\begin{proof}
Since the operator $\mathcal{M}$ generates a strongly continuous
semigroup (Lemma \ref{lem:operator L generates C0}), the first argument
of the lemma is an immediate consequence of the Generation Theorem
of \citep[\S 2.3]{Engel2000}. We now prove that there exist $\alpha>0$
such that $\mathcal{M}+\alpha I$ is dissipative. For a given $f\in\mathcal{D}(\mathcal{M})=\mathcal{D}(\mathcal{L})$
and some $h\in H$ and $\lambda>0$ we have
\[
f-\lambda(\mathcal{M}+\alpha I)f=h\,.
\]
Consequently, multiplying both sides by the sign function of $f$
yields 
\begin{alignat}{1}
\left|f(x)\right| & =f(x)\sign\left(f(x)\right)\nonumber \\
 & =-\lambda\left[g(x)f(x)\right]'\sign\left(f(x)\right)\,dx-\lambda\left[\frac{1}{2}k_{f}(x)+\mu(x)\right]f(x)\sign\left(f(x)\right)\nonumber \\
 & \quad+\lambda\sign\left(f(x)\right)\int_{x}^{x_{1}}\Gamma(x;\,y)k_{f}(y)f(y)\,dy\,dx+\lambda\alpha f(x)+h(x)\sign\left(f(x)\right)\,,\label{eq:function with sign}
\end{alignat}
where function $\sign(f(x))$ is defined as usual with $\sign(0)=0$.
For a given $f\in\mathcal{D}(\mathcal{M})$ the set of points for
which $f$ does not vanish can be written as a finite union of disjoint
open sets $I_{j}=(a_{j},\,b_{j})$, i.e., $f(x)\ne0$ for all $x\in\cup_{j=1}^{n}I_{j}=(0,\,x_{1})$.
\footnote{See also \citep{BanksKappel1989} and \citep{Farkas2008b} for similar
partitioning in dissipativity proofs} On each interval $I_{j}$ the function $f$ can be either strictly
positive or strictly negative. Moreover, on the boundaries we have
$f(a_{j})=0$ and $f(b_{j})=0$ unless $a_{j}=0$ or $b_{j}=x_{1}$.
Then, integrating both sides of (\ref{eq:function with sign}) on
a given interval $I_{j}=(a_{j},\,b_{j})$ we have
\begin{flalign}
\int_{a_{j}}^{b_{j}}\left|f(x)\right|\,dx & \le-\lambda g(b_{j})\left|f(b_{j})\right|+\lambda g(a_{j})\left|f(a_{j})\right|-\lambda\int_{a_{j}}^{b_{j}}\left[-\alpha+\frac{1}{2}k_{f}(x)+\mu(x)\right]\left|f(x)\right|\,dx\nonumber \\
 & \quad+\lambda\int_{a_{j}}^{b_{j}}\int_{x}^{x_{1}}\Gamma(x;\,y)k_{f}(y)\left|f(y)\right|\,dy+\int_{a_{j}}^{b_{j}}\left|h(x)\right|\,dx\,.\label{eq:integral on open intervals}
\end{flalign}
Consequently, by summing (\ref{eq:integral on open intervals}) for
$j=1,\dots,n$ we get 
\begin{flalign*}
\int_{0}^{x_{1}}\left|f(x)\right|\,dx & \le-\lambda g(x_{1})\left|f(x_{1})\right|+\lambda g(0)\left|f(0)\right|-\lambda\int_{0}^{x_{1}}\left[-\alpha+\frac{1}{2}k_{f}(x)+\mu(x)\right]\left|f(x)\right|\,dx\\
 & \quad+\lambda\int_{0}^{x_{1}}\int_{x}^{x_{1}}\Gamma(x;\,y)k_{f}(y)\left|f(y)\right|\,dy\,dx+\int_{0}^{x_{1}}\left|h(x)\right|\,dx\\
 & \le-\lambda\int_{0}^{x_{1}}\left[-\alpha-q(x)+\frac{1}{2}k_{f}(x)+\mu(x)\right]\left|f(x)\right|\,dx\\
 & \quad+\lambda\int_{0}^{x_{1}}k_{f}(y)\left|f(y)\right|\underbrace{\int_{0}^{y}\Gamma(x;\,y)\,dx}_{=1}\,dy+\int_{0}^{x_{1}}\left|h(x)\right|\,dx\\
 & =-\lambda\int_{0}^{x_{1}}\left[-\alpha-q(x)-\frac{1}{2}k_{f}(x)+\mu(x)\right]\left|f(x)\right|\,dx+\int_{0}^{x_{1}}\left|h(x)\right|\,dx\,.
\end{flalign*}
Hence, provided that we have 
\begin{equation}
-\alpha-q(x)-\frac{1}{2}k_{f}(x)+\mu(x)>0\label{eq:conditions on fragmentations}
\end{equation}
for all $x\in I$, it follows that 
\[
\left\Vert f\right\Vert _{1}\le\left\Vert h\right\Vert _{1}=\left\Vert f-\lambda(\mathcal{M}+\alpha I)f\right\Vert _{1}\,.
\]
In fact, if (\ref{eq:condition for stability}) holds true, then there
exists $\alpha>0$ such that $\mathcal{M}+\alpha I$ is dissipative.
Consequently, the result follows immediately from the Lumer-Philips
theorem.
\end{proof}
As a direct consequence of Proposition \ref{prop:linearized stability}
and the above lemma, we summarize the results of this section in form
of the following condition.

\begin{cond}\label{Assume-that-for}Assume that the assumptions ($\mathbf{A}1$)-($\mathbf{A}6$)
hold true. Moreover, assume that 
\[
q(x)+\frac{1}{2}k_{f}(x)-\mu(x)<0
\]
for all $x\in I$, then the zero stationary solution of the flocculation
equation is locally exponentially stable.\end{cond}

\section{\label{sec:Linearized-stability-(instabilit}Linearized instability
and stability criteria for non-trivial steady states}

In this section we present linearized stability results for the non-trivial
stationary solution $p_{*}\ne0$. We first derive conditions for instability
(Section \ref{sub:Linearized-instability}) and then derive conditions
for linear stability (Section \ref{sub:Linearized-stability}).

\subsection{\label{sub:Linearized-instability}Linearized instability}

Recall that, from Proposition \ref{prop:linearized stability}, instability
of the non-trivial stationary solution depends on the spectral properties
of the operator $\mathcal{L}$. Specifically, the spectrum of $\mathcal{L}$
contains at least one point $\lambda_{0}\in\sigma(\mathcal{L})$ satisfying
the instability condition (\ref{eq: instability condition}). Towards
this end (as we did in Section \ref{sub:Instability-of-the trivial }),
we first show that the operator $\mathcal{L}$ has a positive spectral
radius. 
\begin{lemma}
\label{lem:positive specral radius lemma}Assume that the positivity
conditions (\ref{eq:positivity condition 1})-(\ref{eq:positivity condition 2})
hold. Moreover, if the model parameters satisfy the following condition
\begin{equation}
\int_{0}^{x_{1}}\frac{q(x)}{g(x)}\exp\left(-\int_{0}^{x}\frac{\mu(s)+\frac{1}{2}k_{f}(s)+\int_{0}^{x_{1}-s}k_{a}(s,\,y)p_{*}(y)\,dy}{g(s)}\,ds\right)\,dx>1\,,\label{eq:positve spectral radius for nontrivial}
\end{equation}
then the operator $\mathcal{L}$ has a positive spectral radius. \end{lemma}
\begin{proof}
Recall that the operator $\mathcal{L}$ can be written as the sum
of the operators $\mathcal{L}_{1}$, \emph{$\mathcal{L}_{2}$ }and
$\mathcal{L}_{3}$. Moreover, in Lemma \ref{lem:positve semigroup}
we have shown that the operator $\mathcal{L}_{1}$ generates a positive
semigroup and the positivity assumptions (\ref{eq:positivity condition 1})
and (\ref{eq:positivity condition 2}) ensure the positivity of the
operator $\mathcal{L}_{2}+\mathcal{L}_{3}$. Therefore, from Corollary
1.11 of \citep[\S 6]{Engel2000} it follows that the spectral radius
of $\mathcal{L}$ is always greater than the spectral radius of the
operator \emph{$\mathcal{L}_{1}$},\emph{ }i.e., 
\begin{equation}
s(\mathcal{L}_{1})\le s(\mathcal{L}_{1}+\mathcal{L}_{2}+\mathcal{L}_{3})=s(\mathcal{L})\,.\label{eq:instability inequality}
\end{equation}
Conversely, provided that the condition (\ref{eq:positve spectral radius for nontrivial})
holds, the arguments of \citep[Theorem 5.1]{Farkas2007} can be used
to show that the spectral radius of the operator $\mathcal{L}_{1}$
is strictly positive. This result, combined with the inequality (\ref{eq:instability inequality})
implies that the spectral radius of $\mathcal{L}$ is strictly positive. 
\end{proof}
We are now ready to present the main result of this section in the
form of the following condition.

\begin{cond}\label{thm:Instability condition for positive stationary solution}Under
the main assumptions ($\mathbf{A}1$)-($\mathbf{A}6$) and the positivity
conditions (\ref{eq:positivity condition 1})-(\ref{eq:positivity condition 2})
the non-trivial steady state solution of the nonlinear evolution equation
defined in (\ref{eq: agg and growth model}) is unstable if 
\begin{equation}
\int_{0}^{x_{1}}\frac{q(x)}{g(x)}\exp\left(-\int_{0}^{x}\frac{\mu(s)+\frac{1}{2}k_{f}(s)+\int_{0}^{x_{1}-s}k_{a}(s,\,y)p_{*}(y)\,dy}{g(s)}\,ds\right)\,dx>1\,.\label{eq: first instability condition}
\end{equation}
\end{cond}
\begin{proof}
Recall that from the proof of Lemma \ref{lem:positive specral radius lemma}
it follows that
\[
s(\mathcal{L}_{1})\le s(\mathcal{L})\,,
\]
where the operators $\mathcal{L}_{1}$ and $\mathcal{L}$ are defined
in (\ref{eq:linearized operator L}). Note that \citet{Farkas2007}
have shown that the operator $\mathcal{L}_{1}$ has a positive spectral
radius provided that the condition (\ref{eq: first instability condition})
holds true. Consequently, if the condition (\ref{eq: first instability condition})
holds true, it follows that the operator $\mathcal{L}$ has a positive
spectral radius. Then from Proposition \ref{lem:eventually compact}
and Remark \ref{rem:growth bound in spectrum} it follows that $s(\mathcal{L})\in\sigma_{D}(\mathcal{L})$.
Moreover, Proposition \ref{lem:eventually compact} together with
\citep[Remark 4.8]{Webb1985} imply that\emph{ $\alpha$}-growth bound
of $\mathcal{L}$ is equal to negative infinity, $\omega_{1}(\mathcal{L})=-\infty$.
Therefore, in Proposition \ref{prop:linearized stability}, choosing
$\lambda_{0}$ equal to the eigenvalue corresponding to $s(\mathcal{L})$
yields
\[
\max\left\{ \omega_{1}(\mathcal{L}),\,\sup_{\lambda\in\sigma_{D}(\mathcal{L})\backslash\{\lambda_{0}\}}\text{Re}\,\lambda\right\} =\sup_{\lambda\in\sigma_{D}(\mathcal{L})\backslash\{\lambda_{0}\}}\text{Re}\,\lambda<\text{Re}\,\lambda_{0}\,,
\]
and implies that the non-trivial stationary solution $p_{*}$ of the
nonlinear evolution equation defined in (\ref{eq: agg and growth model})
is unstable.
\end{proof}

\begin{rem}Let $S\subset L^{1}(I)$ be the set of non-trivial stationary
solutions and $S_{1}$ (subset of $S$) denote the set of non-trivial
stationary solutions existence of which guaranteed by Theorem \ref{thm:positive steady state existence}.
For a stationary solution $p_{*}\in S_{1}$ the modeling terms need
to satisfy the conditions ($\mathbf{C}1$) and ($\mathbf{C}2$). Consequently,
plugging $x=0$ into ($\mathbf{C}$2) yields the inequality 
\[
\int_{0}^{x_{1}}\frac{q(x)}{g(x)}\,dx\le1\,.
\]
Conversely, the instability condition (\ref{eq: first instability condition})
implies that 
\[
1<\int_{0}^{x_{1}}\frac{q(x)}{g(x)}\exp\left(-\int_{0}^{x}\frac{\mu(s)+\frac{1}{2}k_{f}(s)+\int_{0}^{x_{1}-s}k_{a}(s,\,y)p_{*}(y)\,dy}{g(s)}\,ds\right)\,dx\le\int_{0}^{x_{1}}\frac{q(x)}{g(x)}\,dx\,,
\]
which contradicts the existence condition ($\mathbf{C}$2). This in
turn implies that stationary solutions in the set $S_{1}$ do not
satisfy the instability condition. However, we note that $S_{1}$
is only subset of $S$, and thus the results of this subsection is
only valid for non-trivial stationary solutions in the set $S\backslash S_{1}$.\end{rem}

\subsection{\label{sub:Linearized-stability}Linearized stability}

In Section \ref{sub:Linearized-instability} we have shown that the
spectrum of the operator $\mathcal{L}$ is not empty. This result,
together with Proposition \ref{lem:positve semigroup} and Remark
\ref{rem:growth bound in spectrum} imply that the spectral radius
of $\mathcal{L}$ is one of the eigenvalues of the operator $\mathcal{L}$,
so it is sufficient to show that all the eigenvalues of $\mathcal{L}$
have a negative real part. However, to the best of our knowledge,
the eigenvalue problem 
\[
\mathcal{L}[\phi]=\lambda\phi
\]
does not have an explicit solution. This forces us to utilize the
positive perturbation method \citet{Farkas2010} to locate the dominant
eigenvalue of $\mathcal{L}$. This method relies on the fact that
compact perturbations do not change the essential spectrum of a semigroup.
Towards this end we will perturb the operator $\mathcal{L}=\mathcal{L}_{1}+\mathcal{L}_{2}+\mathcal{L}_{3}$
(the operators $\mathcal{L}_{1}$, $\mathcal{L}_{2}$ and $\mathcal{L}_{3}$
are defined in (\ref{eq:operator L1})-(\ref{eq:operators L2 and L3}))
by a positive compact operator so that we can identify the point spectrum
of the resulting operator. 
\begin{lemma}
\label{lem:Operator C is positive}Let us define the operator\emph{
$\mathcal{C}$ as} 
\[
\mathcal{C}[f](x)=c_{1}\int_{0}^{x_{1}}f(y)\,dy,\,
\]
where $c_{1}=\left\Vert k_{a}\cdot p_{*}\right\Vert _{\infty}+\left\Vert \Gamma\cdot k_{f}\right\Vert _{\infty}$.
Then the operator $\mathcal{C}-\mathcal{L}_{2}-\mathcal{L}_{3}$ is
positive and compact.\end{lemma}
\begin{proof}
It is easy to see that\emph{ $\mathcal{C}-\mathcal{L}_{2}-\mathcal{L}_{3}$
is a positive operator, i.e.,} 
\begin{alignat*}{1}
\mathcal{C}[f](x)-\mathcal{L}_{2}[f](x)-\mathcal{L}_{3}[f](x) & \ge\int_{0}^{x}\left[\left\Vert k_{a}\cdot p_{*}\right\Vert _{\infty}-k_{a}(x-y,\,y)p_{*}(x-y)\right]f(y)\,dy\\
 & \quad+\int_{x}^{x_{1}}\left[\left\Vert \Gamma\cdot k_{f}\right\Vert _{\infty}-\Gamma(x;\,y)k_{f}(y)\right]f(y)\,dy\ge0\quad\forall f\in\left(L^{1}(I)\right)_{+}\,.
\end{alignat*}
Conversely, $\mathcal{C}$ is a bounded linear operator of rank one,
hence it is compact. Then the compactness of $\mathcal{C}-\mathcal{L}_{2}-\mathcal{L}_{3}$
follows from compactness of the operators $\mathcal{L}_{2}$ and $\mathcal{L}_{3}$
(see Lemma \ref{lem:L2 is compact}). 
\end{proof}

Now define the perturbed operator $\mathcal{P}$ as $\mathcal{P}:=\mathcal{L}+\mathcal{C}-\mathcal{L}_{2}-\mathcal{L}_{3}=\mathcal{L}_{1}+\mathcal{C}$.
Then the eigenvalue problem for the operator \emph{$\mathcal{P}$
}reads as
\begin{equation}
\lambda f-\mathcal{P}[f]=\lambda f-\mathcal{L}_{1}[f]-\mathcal{C}[f]=0\,.\label{eq:perturbed eigenvalue problem}
\end{equation}
This equation can be solved implicitly as
\begin{equation}
f(x)=U_{1}\frac{1}{T(\lambda,\,x)g(x)}+U_{2}\frac{c_{1}}{g(x)T(\lambda,\,x)}\int_{0}^{x}T(\lambda,\,s)\,ds\,,\label{eq:implicit solution}
\end{equation}
where 
\[
U_{1}=\int_{0}^{x_{1}}q(y)f(y)\,dy,\qquad U_{2}=\int_{0}^{x_{1}}f(y)\,dy
\]
and 
\[
T(\lambda,\,x)=\exp\left(\int_{0}^{x}\frac{\lambda+A(y)}{g(y)}\,dy\right)\,.
\]
Integrating the equation (\ref{eq:implicit solution}) on $I$ yields
one equation for solving for $U_{1}$ and $U_{2}$. Moreover, multiplying
the equation (\ref{eq:implicit solution}) by $q(x)$ and integrating
over the interval $I$ we obtain the second equation for solving for
$U_{1}$ and $U_{2}$. Consequently, these two equations can be summarized
in the following linear system, 
\begin{equation}
\begin{cases}
U_{1}A_{11}(\lambda)+U_{2}\left(A_{12}(\lambda)-1\right)=0\\
U_{1}(A_{21}(\lambda)-1)+U_{2}A_{22}(\lambda)=0
\end{cases}\,,\label{eq:the linear system}
\end{equation}
where
\begin{alignat*}{1}
A_{11}(\lambda)=\int_{0}^{x_{1}}\frac{1}{T(\lambda,\,x)g(x)}\,dx,\quad A_{12}(\lambda)=\int_{0}^{x_{1}}\frac{c_{1}}{g(x)T(\lambda,\,x)}\int_{0}^{x}T(\lambda,\,s)\,ds\,dx\,,\\
A_{21}(\lambda)=\int_{0}^{x_{1}}\frac{q(x)}{T(\lambda,\,x)g(x)}\,dx,\quad A_{22}(\lambda)=\int_{0}^{x_{1}}\frac{c_{1}q(x)}{g(x)T(\lambda,\,x)}\int_{0}^{x}T(\lambda,\,s)\,ds\,dx\,.
\end{alignat*}
If the eigenvalue problem (\ref{eq:perturbed eigenvalue problem})
has a non-zero solution, then there is non-zero vector $(U_{1},\,U_{2})$
satisfying the linear system (\ref{eq:the linear system}). On the
other hand, if there is non-zero vector $(U_{1},\,U_{2})$ satisfying
the linear system (\ref{eq:the linear system}), then the eigenvalue
problem has a non-zero solution. Hence $\lambda\in\mathbb{C}$ is
an eigenvalue value of the operator $\mathcal{P}$ if and only if
\begin{equation}
K(\lambda)=\det\left(\begin{array}{cc}
A_{11}(\lambda) & A_{12}(\lambda)-1\\
A_{21}(\lambda)-1 & A_{22}(\lambda)
\end{array}\right)=A_{11}(\lambda)A_{22}(\lambda)-\left(1-A_{12}(\lambda)\right)\left(1-A_{21}(\lambda)\right)=0\,.\label{eq:Characteristic function}
\end{equation}
In structured population dynamics the function $K$ is often referred
as a characteristic function of eigenvalues of an operator, and similar
characteristic functions have been derived in \citep{Equilibria1983,Farkas,Farkas2005,Farkas2007,Farkas2010}.
The main advantage of having the characteristic function $K$ is the
task of locating the dominant eigenvalue value of the operator $\mathcal{L}$
reduces to locating the roots of the function $K$. Hence, in the
following lemma we show that under certain conditions on the model
parameters all the roots of the characteristic function $K$ lie in
the left half of the complex plane. 
\begin{lemma}
\label{lem:negative root}Under the conditions
\begin{equation}
A_{12}(0)<1,\qquad A_{21}(0)<1\label{eq:stability condition}
\end{equation}
and 
\begin{equation}
K(0)<0\label{eq:stability condition K(0)}
\end{equation}
the function $K$ does not have any roots with non-negative real part.
Furthermore, the function $K$ has at least one negative real root. \end{lemma}
\begin{proof}
It is straightforward to see that 
\[
A_{11}(\lambda)=A_{12}(\lambda)=A_{21}(\lambda)=A_{22}(\lambda)=0\text{ as }\lambda\to\infty\,,
\]
so 
\[
\lim_{\lambda\to\infty}K(\lambda)=-1\,.
\]
Moreover, observe that for $i=1,\,2$ and $j=1,\,2$ the functions
$A_{ij}\,:\,\mathbb{R}\to\mathbb{R}_{+}$ are non-increasing, i.e.,
\[
\partial_{\lambda}A_{ij}(\lambda)\le0\,.
\]
Consequently, for $\lambda\ge0$ from (\ref{eq:stability condition})
we have
\[
1-A_{12}(\lambda)\ge1-A_{12}(0)>0
\]
 and 
\[
1-A_{21}(\lambda)\ge1-A_{21}(0)>0\,.
\]
Conversely, differentiating $K(\lambda)$ for $\lambda\ge0$ yields
\[
K'=A'_{11}A_{22}+A_{11}A'_{22}+A'_{12}\underbrace{(1-A_{21})}_{>0}+A'_{21}\underbrace{(1-A_{12})}_{>0}\le0\,.
\]
Thus the function $K$ restricted to real numbers is non-increasing.
This in turn together with the condition (\ref{eq:stability condition K(0)})
implies that the function $K$ does not have any positive real root. 

Now for the sake of a contradiction, assume that there is $\lambda_{1}=a-b\mathrm{i}\in\mathbb{C}$
with $a\ge0$ and $b\ne0$ such that 
\begin{equation}
K(\lambda_{1})=0\,.\label{eq:Complex root}
\end{equation}
Let us define 
\[
G(x)=\int_{0}^{x}\frac{1}{g(y)}\,dy\,,
\]
then for $\lambda=a-b\mathrm{i}$ we have 
\[
A_{11}(\lambda_{1})=\int_{0}^{x_{1}}\frac{\cos\left[bG(x)\right]}{T(a,\,x)g(x)}\,dx+\mathrm{i}\int_{0}^{x_{1}}\frac{\sin\left[bG(x)\right]}{T(a,\,x)g(x)}\,dx\,.
\]
This in turn implies that 
\begin{equation}
-A_{11}(a)\le\text{Re }A_{11}(\lambda_{1})\le A_{11}(a)\,.\label{eq:inequalities for A_ij}
\end{equation}
Analogous arguments yields similar inequalities for $A_{12},\,A_{21}$
and $A_{22}$. On the other hand, if (\ref{eq:Complex root}) holds
true then using (\ref{eq:inequalities for A_ij}) it follows that
\begin{alignat}{1}
A_{11}(a)A_{12}(a) & =\left|A_{11}(\lambda_{1})A_{12}(\lambda_{1})\right|=\left|\left(1-A_{12}(\lambda_{1})\right)\left(1-A_{21}(\lambda_{1})\right)\right|\nonumber \\
 & \ge\left|1-\text{Re }A_{12}(\lambda_{1})\right|\left|1-\text{Re }A_{21}(\lambda_{1})\right|\nonumber \\
 & \ge\left(1-A_{12}(a)\right)\left(1-A_{21}(a)\right)\ge0\label{eq:contradiction}
\end{alignat}
Since $K(\lambda)$ is non-increasing for $\lambda\ge0$, from (\ref{eq:stability condition K(0)})
we have 
\begin{equation}
K(a)\le K(0)<0\,.\label{eq:contradiction 2}
\end{equation}
The equation (\ref{eq:contradiction 2}) is equivalent to 
\[
A_{11}(a)A_{12}(a)<\left(1-A_{12}(a)\right)\left(1-A_{21}(a)\right)\,,
\]
which obviously contradicts the equation (\ref{eq:contradiction}).
Hence, the function cannot have a complex root with a non-negative
real part. 

To establish the last statement of the lemma, observe that the function
$1-A_{12}(\lambda)$ is continuous and non-decreasing with 
\[
\lim_{\lambda\to-\infty}1-A_{12}(\lambda)=-\infty\,.
\]
Conversely, from the condition (\ref{eq:stability condition}) we
have $1-A_{12}(\lambda)>0$ for $\lambda\ge0$. Thus by Intermediate
Value Theorem there is $\lambda_{0}<0$ such that $1-A_{12}(\lambda_{0})=0$.
Consequently, evaluating $K$ at $\lambda=\lambda_{0}$ yields
\[
K(\lambda_{0})=A_{11}(\lambda_{0})A_{22}(\lambda_{0})-\left(1-A_{12}(\lambda_{0})\right)\left(1-A_{21}(\lambda_{0})\right)=A_{11}(\lambda_{0})A_{22}(\lambda_{0})\ge0\,.
\]
Hence, the function $K$ has at least one negative real root, which
completes the proof of the lemma.
\end{proof}
With the above lemma in hand, we can now state the main result of
this subsection in the form of the following condition.

\begin{cond}\label{thm:Stability condition for positive stationary solution}Suppose
that the conditions 
\[
K(0)=A_{11}(0)A_{22}(0)-\left(1-A_{12}(0)\right)\left(1-A_{21}(0)\right)<0\,,
\]
 
\[
A_{12}(0)=\left(\left\Vert k_{a}\cdot p_{*}\right\Vert _{\infty}+\left\Vert \Gamma\cdot k_{f}\right\Vert _{\infty}\right)\int_{0}^{x_{1}}\frac{1}{g(x)}\int_{0}^{x}\exp\left(-\int_{s}^{x}\frac{A(s)}{g(s)}\right)\,ds\,dx<1
\]
 and 
\[
A_{21}(0)=\int_{0}^{x_{1}}\frac{q(x)}{g(x)}\exp\left(-\int_{0}^{x}\frac{A(s)}{g(s)}\,ds\right)\,dx<1
\]
 hold true. Then, the non-trivial steady state solution $p_{*}$ is
linearly exponentially stable.\end{cond}
\begin{proof}
Lemma \ref{lem:negative root} implies that the operator $\mathcal{L}_{1}+\mathcal{C}$
has negative spectral radius, 
\[
s(\mathcal{L}_{1}+\mathcal{C})<0\,.
\]
Conversely, from \citet[\S 6, Corollary 1.11]{Engel2000}, Proposition
\ref{lem:positve semigroup} and Lemma \ref{lem:Operator C is positive}
it follows that 
\[
s(\mathcal{L})=s(\mathcal{L}_{1}+\mathcal{L}_{2}+\mathcal{L}_{3})\le s(\mathcal{L}_{1}+\mathcal{L}_{2}+\mathcal{L}_{3}+\mathcal{C}-\mathcal{L}_{2}-\mathcal{L}_{3})=s(\mathcal{L}_{1}+\mathcal{C})<0\,.
\]
Consequently, from \citep[\S 6, Theorem 1.15 ]{Engel2000} it follows
that 
\[
\omega_{0}(\mathcal{L})=s(\mathcal{L})<0\,.
\]
Hence, Proposition \ref{prop:linearized stability} yields that the
non-trivial steady state solution $p_{*}$ is linearly asymptotically
stable.
\end{proof}

\section{\label{sec:Illustration-of-the results}Illustration of the results}

In this section, we illustrate the theoretical development of the
paper by giving explicit examples. 

\begin{exmp}\label{exa:zero solution example}\end{exmp}First, observe
that $p(t,\,x)\equiv0$ is always a stationary solution of the flocculation
model. Therefore, we only have to worry about the stability conditions
derived in Section \ref{sec:zero Linearized-stability-results}. For
the maximal floc size, removal and renewal rates we choose terms similar
to that of \citep{Ackleh2014}
\[
\mu(x)=1,\quad q(x)=b(x+1)
\]
where $b$ has yet to be chosen. Moreover, we assume that growth and
fragmentation rates are proportional to the size of the floc, 
\[
g(x)=x+1\text{ and }k_{f}(x)=2x\,.
\]
Since the stability of the zero stationary solution does not depend
on the aggregation rate and the post-fragmentation density function,
we only assume that the functions $k_{a}$ and $\Gamma$ satisfy the
main assumptions (\textbf{$\mathbf{A}2$}) and ($\mathbf{A}5$), respectively.
Consequently, plugging in this values into the instability condition
(\ref{eq:zero instability condition}) yields
\[
\int_{0}^{1}\frac{q(x)}{g(x)}\exp\left(-\int_{0}^{x}\frac{\mu(s)+\frac{1}{2}k_{f}(s)}{g(s)}\,ds\right)\,dx=b\int_{0}^{1}\exp\left(-\int_{0}^{x}1\,ds\right)\,dx=b(1-e^{-1})>1\,.
\]
Thus, provided that $b>2$, the zero stationary solution of the flocculation
model is unstable. 

Conversely, plugging in the above parameters in the stability condition
(\ref{eq:condition for stability}) one obtains
\[
q(x)+\frac{1}{2}k_{f}(x)-\mu(x)=bx+b-1<0\,.
\]
Therefore, provided that we have $0<b<\frac{1}{2}$, we can guarantee
the local exponential stability of the zero stationary solution.

\begin{exmp}\end{exmp}In contrast to the zero stationary solution,
positive stationary solutions do not always exist. Based on the analysis
of Section \ref{sec:Existence-of-a postive}, if the model parameters
satisfy the existence conditions $(\mathbf{C}1)$ and $(\mathbf{C}2)$
of Section \ref{sec:Existence-of-a postive}, then the flocculation
model possesses at least one nontrivial stationary solution. Towards
this end, we choose an exponential growth rate and\emph{ a uniform
probability distribution} for the post-fragmentation density function,
\[
g(x)=\exp(-ax),\quad\Gamma(x;\,y)=\frac{\chi_{(0,\,y]}(x)}{y}\,.
\]
For other model parameters we choose linear rates,
\[
q(x)=b(x+1),\quad k_{f}(x)=cx,\quad\mu(x)=\frac{c}{2}x\,.
\]
A straightforward computation shows that the above model parameters
satisfy the first existence conditions $(\mathbf{C}1)$. Furthermore,
plugging in the above rates into the second existence condition $(\mathbf{C}2)$
yields
\begin{equation}
c\frac{e^{a}}{a}+b\frac{a+e^{a}-1-2ae^{a}}{a^{2}}\le1\,.\label{eq:second existence condition}
\end{equation}
Hence, provided that the constants $a,\,b$ and $c$ satisfy the inequality
(\ref{eq:second existence condition}), there exists a positive integrable
stationary solution $p_{*}$. Before we carry out the stability analysis
for this stationary solution $p_{*}$, the above functions should
also satisfy the regularity conditions of Section \ref{sec:Principle-of-linearized}.
However, the regularity conditions of Section \ref{sec:Principle-of-linearized}
depend heavily on the explicit form of the positive stationary solution
in $k_{a}(x,\,y)p_{*}(x)$. Since our analysis does not provide exact
form of the positive stationary solution, we further assume that the
stationary solution satisfies $p_{*}(x)>0$ for all $x\in I$. Consequently,
we tailor the aggregation kernel such that it satisfies the positivity
condition (\ref{eq:positivity condition 2}) and thus we proceed by
choosing the aggregation kernel as
\[
k_{a}(x,\,y)=\frac{d\,(1-x)(1-y)}{p_{*}(x)p_{*}(y)}\,,
\]
where the constant $d$ has to be chosen sufficiently small such that
\[
c\min_{x\in I}p_{*}(x)>d\,.
\]
This form of the aggregation kernel indeed solves the problem and
all the regularity conditions of Section \ref{sec:Principle-of-linearized}
are all satisfied.

We are now in a position to derive conditions for the stability (and
instability) of the positive stationary solution $p_{*}(x)$. Towards
this end, plugging the parameters into the instability conditions
stated in Condition \ref{thm:Instability condition for positive stationary solution}
and using the inequality (\ref{eq:second existence condition}) yields{\small{}
\[
\int_{0}^{1}\frac{q(x)}{g(x)}\exp\left(-\int_{0}^{x}\frac{\mu(s)+\frac{1}{2}k_{f}(s)+\int_{0}^{1-s}k_{a}(s,\,y)p_{*}(y)\,dy}{g(s)}\,ds\right)\,dx\le\int_{0}^{1}\frac{q(x)}{g(x)}\,dx=b\frac{a+e^{a}-1-2ae^{a}}{a^{2}}\le1\,.
\]
}This in turn implies that these parameters do not satisfy the instability
condition. However, since our instability conditions are only sufficient
conditions, this does not imply that the positive stationary solution
$p_{*}$ is always stable.

For the stability of the positive stationary solution, we plug in
the above parameters to the stability conditions stated in Condition
\ref{thm:Stability condition for positive stationary solution} and
obtain 
\begin{equation}
A_{12}(0)\le c_{1}\frac{e^{a}(a-1)+1}{a^{2}}<1,\,A_{21}(0)\le b\frac{a+e^{a}-1-2ae^{a}}{a^{2}}<1\,,\label{eq:Stability condition 2}
\end{equation}
where 
\[
c_{1}=\left\Vert k_{a}\cdot p_{*}\right\Vert _{\infty}+\left\Vert \Gamma\cdot k_{f}\right\Vert _{\infty}=\frac{d}{\min_{x\in I}p_{*}(x)}+c<2c\,.
\]
Moreover, the last stability condition, $K(0)<0$, yields 
\begin{alignat*}{1}
A_{11}(0)\cdot A_{22}(0) & <c_{1}\frac{(e^{a}-1)\left(a-2+(2a^{2}-3a+2)e^{a}\right)}{a^{4}}\\
 & <\left(1-c_{1}\frac{e^{a}(a-1)+1}{a^{2}}\right)\left(1-b\frac{a+e^{a}-1-2ae^{a}}{a^{2}}\right)\\
 & <(1-A_{12}(0))(1-A_{21}(0))\,.
\end{alignat*}
Conversely, since $c_{1}<2c$ this yields another inequality for the
parameters $a,\,b$ and $c$,
\begin{equation}
2c\frac{(e^{a}-1)\left(a-2+(2a^{2}-3a+2)e^{a}\right)}{a^{4}}<\left(1-2c\frac{e^{a}(a-1)+1}{a^{2}}\right)\left(1-b\frac{a+e^{a}-1-2ae^{a}}{a^{2}}\right)\,.\label{eq:stability condition-1}
\end{equation}
Thus, provided that we choose the parameters $a,\,b$ and $c$ such
that the inequalities (\ref{eq:second existence condition})-(\ref{eq:stability condition-1})
hold, we can guarantee local exponential stability of this positive
stationary solution. For the convenience of the readers, in Figure
\ref{fig:Stability-regions-for} we have illustrated the stability
region for the parameters $a,\,b$ and $c$. Observe that for larger
values of the parameter $a$ the stability region for the parameters
$b$ and $c$ significantly shrinks. As the value of the parameter
$a$ increases the growth rate decreases. Thus, from biological point
of view, in order to keep the non-trivial stationary solution stable
the values of the parameters $b$ and $c$, associated to removal,
fragmentation and renewal rates, should also decrease.
\begin{figure}
\begin{centering}
\includegraphics[scale=0.5]{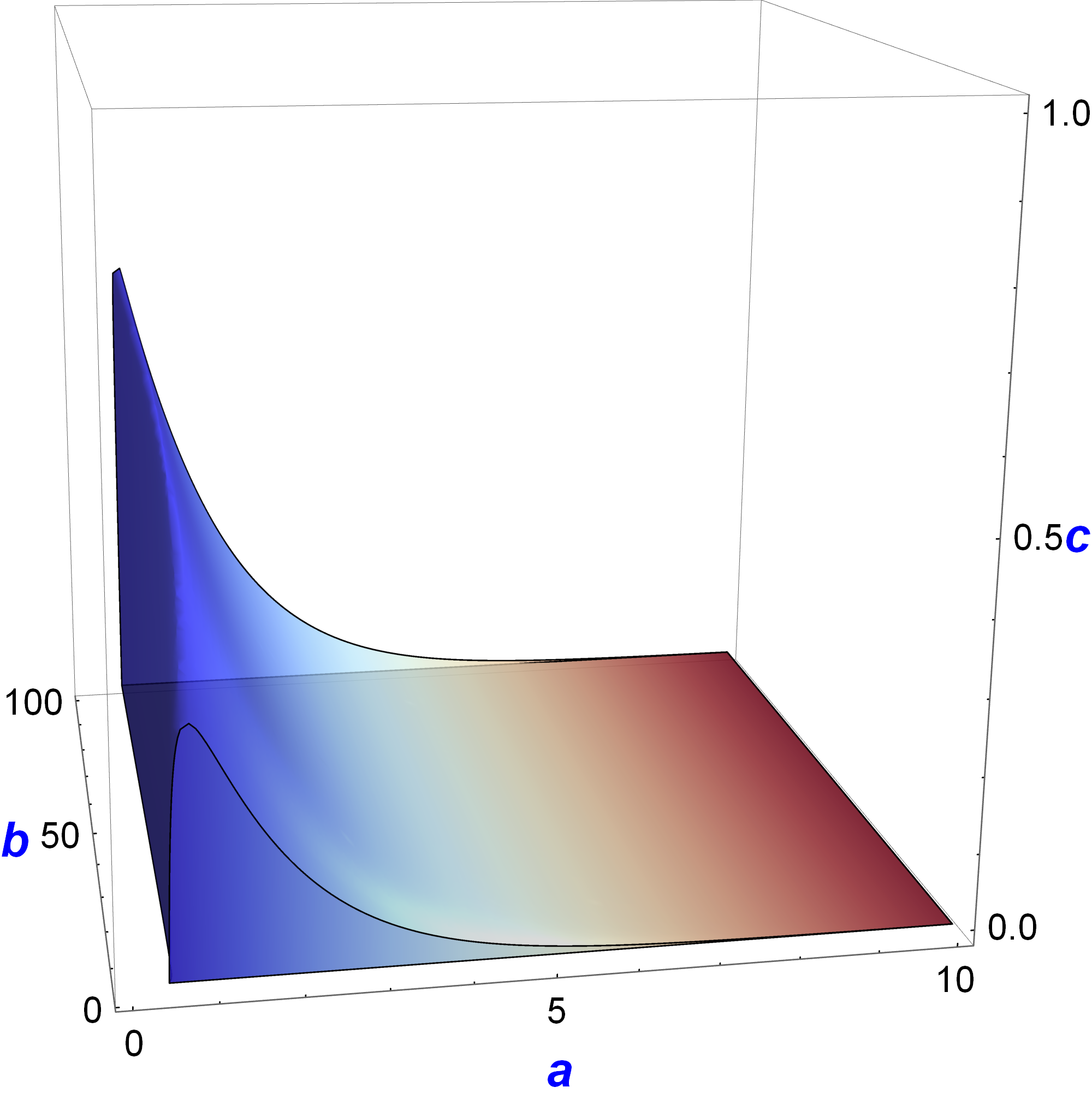}
\par\end{centering}

\protect\caption{\label{fig:Stability-regions-for}Stability region for the parameters
$a,\,b$ and $c$. Positive $a\,,b$ and $c$ values lying inside
the solid assure local exponential stability of the positive stationary
solution.}
\end{figure}

\section{\label{sec:Concluding-remarks}Concluding remarks}

Our primary motivation in this paper is to investigate the ultimate
behavior of solutions of a generalized size-structured flocculation
model. The model accounts for a broad range of biological phenomena
(necessary for survival of a community of microorganism in a suspension)
including growth, aggregation, fragmentation, removal due to predation,
and gravitational sedimentation. Moreover, the number of cells that
erode from a floc and enter the single cell population is modeled
with McKendrick-von Foerster type renewal boundary equation. Although
it has been shown that the model has a unique positive solution, to
the best of our knowledge, the large time behavior of those solutions
has not been studied.

Using a fixed point theorem we showed that under relatively weak restrictions,
which balance removal, growth, fragmentation and renewal rates, the
flocculation model possesses a non-trivial stationary solution (in
addition to the trivial stationary solution). We used the principle
of linearized stability for nonlinear evolution equations to linearize
the problem around the stationary solution. This allowed us to infer
stability of the stationary solutions by the spectral properties of
the linearized problem. We then used the rich theory developed for
semigroups, to derive the stability and instability conditions for
the zero stationary solution. To derive instability conditions for
the non-trivial stationary solution, we employed an eigenvalue localization
method based on the well-known Krein-Rutman theorem. Lastly, we used
compactness and positivity arguments to derive conditions for local
stability of the non-trivial stationary solutions. 

Lastly, even though we showed that the flocculation model has at least
one non-trivial stationary solution, our analysis does not state these
stationary solutions explicitly. The nonlinearity introduced to the
model by Smoluchowski coagulation equations, makes the task of finding
explicit stationary solutions challenging even for constant model
parameters. On the other hand, when only Smoluchowski coagulation
equation is considered, it has been shown that the model has closed
form self-similar solutions for constant and additive aggregation
kernels \citep{Menon2004,Wattis2006a}. Perhaps, under some conditions
on the initial distribution and model parameters, solutions of the
flocculation model also converge to self-similar profiles. Hence,
as a future research, we plan to further investigate self-similar
solutions of the flocculation model.

\section*{Acknowledgements}

Funding for this research was supported in part by grants NSF-DMS
1225878 and NIH-NIGMS 2R01GM069438-06A2. 

\bibliographystyle{apalike}
\bibliography{bib_file}

\end{document}